\documentclass[10pt]{article}
\usepackage[utf8]{inputenc}

\usepackage[margin=1in]{geometry} 
\usepackage{graphicx}

\usepackage{booktabs} 
\usepackage{array} 
\usepackage{paralist} 
\usepackage{verbatim}
\usepackage{mathrsfs}
\usepackage{amssymb}
\usepackage{amsthm}
\usepackage{amsmath,amsfonts,amssymb}
\usepackage{esint}
\usepackage{graphics}
\usepackage{enumerate}
\usepackage{mathtools}
\usepackage{xfrac}
\usepackage{bbm}
\usepackage{caption}
\usepackage{subcaption}
\usepackage{tikz-cd}
\usepackage{adjustbox}

\usepackage[colorlinks=true, pdfstartview=FitV, linkcolor=blue, citecolor=blue, urlcolor=blue]{hyperref}

\numberwithin{equation}{section}
\numberwithin{figure}{section}

\newtheorem{theorem}{Theorem}[section]

\newtheorem{proposition}[theorem]{Proposition}
\newtheorem{lemma}[theorem]{Lemma}
\theoremstyle{definition}
\newtheorem{definition}[theorem]{Definition}

\newtheorem{remark}[theorem]{Remark}

\newcommand*{\N}{\ensuremath{\mathbb{N}}}

\newcommand*{\Z}{\ensuremath{\mathbb{Z}}}

\newcommand*{\R}{\ensuremath{\mathbb{R}}}

\renewcommand*{\tilde}{\widetilde}

\newcommand{\E}{\mathbb{E}}

\DeclareSymbolFont{boldoperators}{OT1}{cmr}{bx}{n}
\SetSymbolFont{boldoperators}{bold}{OT1}{cmr}{bx}{n}
\edef\bar{\unexpanded{\protect\mathaccentV{bar}}\number\symboldoperators16}

\definecolor{labelkey}{rgb}{0,0,1}

\usepackage{titlesec}

\newcommand{\addperiod}[1]{#1.}
\titleformat{\section}
   {\centering\normalfont\Large}{\thesection.}{0.5em}{}
\titleformat*{\subsection}{\bfseries}
\titleformat{\subsubsection}[runin]
  {\normalfont\bfseries}
  {\thesubsubsection.}
  {0.5em}
  {\addperiod}
\titleformat*{\subsubsection}{\bfseries}
\titleformat*{\paragraph}{\bfseries}
\titleformat*{\subparagraph}{\large\bfseries}

\title{Quantitative delocalisation\\ for the Gaussian and $q$-SOS long-range chains}

\author{Loren Coquille\footnote{Universit\'e Grenoble Alpes, CNRS, Institut Fourier, F-38000 Grenoble, France; loren.coquille@univ-grenoble-alpes.fr}, 
	Paul Dario\footnote{Laboratoire AGM, CNRS UMR 8088, Cergy Paris Universit\'e, Cergy-Pontoise, France; paul.dario@cyu.fr},
	Arnaud Le Ny\footnote{Laboratoire d'Analyse et de Mathématiques Appliquées,  UMR CNRS 8050, UPEC, 61 Avenue du G\'en\'eral de Gaulle,  94010 Cr\'eteil cedex, France; arnaud.le-ny@u-pec.fr}
}

\date{ \today}

\usepackage[nottoc,notlot,notlof]{tocbibind}

\begin{document}

\maketitle

\begin{abstract}
The goal of this article is to study quantitatively the localisation/delocalisation properties of the integer-valued Gaussian chain with long-range interactions. Specifically, we consider the integer-valued Gaussian chain of length $N$, with Dirichlet boundary condition, range exponent $\alpha \in (1 , \infty)$ and inverse temperature $\beta \in (0,\infty)$, and show that:
\begin{itemize}
    \item For $\alpha \in (2 ,3)$ and $\beta \in (0 , \infty)$, the fluctuations of the chain are at least of order $N^{\frac{1}{2}(\alpha - 2)}$;
    \item For $\alpha = 3$ and $\beta \in (0 , \infty)$, the fluctuations of the chain are of order $\sqrt{N / \ln N}$ (sharp upper and lower bounds up to multiplicative constants are derived).
\end{itemize}
Combined with the results of Kjaer-Hilhorst~\cite{KH82}, Fr\"{o}hlich-Zegarlinski~\cite{frohlich1991phase} and Garban~\cite{garban2023invisibility}, these estimates provide an (almost) complete picture for the localisation/delocalisation of the integer-valued Gaussian chain. The proofs are based on graph surgery techniques which have been recently developed by van Engelenburg-Lis~\cite{van2023elementary, van2023duality} and Aizenman-Harel-Peled-Shapiro~\cite{aizenman2021depinning} to study the phase transitions of two dimensional integer-valued height functions (and of their dual spin systems).

Additionally, by combining the previous strategy with a technique introduced by Sellke~\cite{sellke2024localization}, we are able extend the method to study the $q$-SOS long-range chain with exponent $q \in (0 , 2)$ and show that, for any inverse temperature $\beta\in (0, \infty)$ and any range exponent $\alpha \in (1 , \infty)$:
\begin{itemize}
    \item The fluctuations of the chain are at least of order $N^{\frac{1}{q}(\alpha -2) \wedge \frac{1}{2}}$;
    \item The fluctuations of the chain are at most of order $N^{\left( \frac{1}{q}\alpha - 1 \right) \wedge \frac 12}$.
\end{itemize}
\end{abstract}

\bigskip

\vspace{0.3cm}


\newpage
\setcounter{tocdepth}{3}
\tableofcontents
\newpage

\section{Introduction}

\subsection{Setting and main results} \label{sec:sec1.2intro}

\subsubsection{integer-valued Gaussian long-range chain} \label{sec:introDGC}

The integer-valued Gaussian long-range chain is a model of discrete interfaces in one dimension with long-range interactions. Formally, it is a model of random interfaces (or height functions) where the interfaces are modelled by functions $\varphi : \Z \to \Z$ which are sampled according to the (formal) Gibbs measure
\begin{equation*}
    \mu(\{ \varphi \}) :=\frac{1}{Z} \exp \left( - \beta \sum_{i , j \in \Z} \frac{\left| \varphi(i) - \varphi(j) \right|^2}{|i - j|^\alpha} \right),
\end{equation*}
where $\beta \in (0 , \infty)$ is the inverse temperature and $\alpha \in (1 , \infty)$ is the range exponent. Following this formalism, we will frequently refer to the value $\varphi (v)$ as the height of the interface $\varphi : \Z \to \Z$ at the vertex $v \in \Z$.

To be more rigorous, we introduce the finite-volume version of the model with Dirichlet boundary condition studied in this manuscript.

\begin{definition}[Finite-volume integer-valued Gaussian long-range chain] \label{def.def1.1}
For $N \in \N$, $\beta \in (0 , \infty)$ and $\alpha \in (1 , \infty)$, we define the integer-valued Gaussian long-range chain of length $N$ at inverse temperature $\beta$ and with range exponent $\alpha$ to be the probability distribution on the set of functions 
$$\Omega_N := \left\{ \varphi : \Z \to \Z \, : \, \forall \, k \notin \{ -N , \ldots ,N\}, \, \varphi(k) = 0 \right\}$$
given be the formula, for any $\varphi \in \Omega_N,$
\begin{equation} \label{def.IVGFFlongrange}
    \mu_{N , \beta , \alpha}(\{ \varphi \})  = \frac{1}{Z_{N , \beta , \alpha}} \exp \left( - \beta \sum_{i , j \in \Z} \frac{\left| \varphi(i) - \varphi(j) \right|^2}{|i - j|^\alpha} \right),
\end{equation}
where $Z_{N , \beta , \alpha}$ is the normalizing constant chosen so that $\mu_{N , \beta , \alpha}$ is a probability distribution. We denote by $\mathrm{Var}_{N , \beta , \alpha}$ the variance with respect to $\mu_{N , \beta , \alpha}$.
\end{definition}

In this article, we study the localisation/delocalisation properties of the integer-valued Gaussian chain, that is, we investigate the variance of the height of the interface at the vertex $0$ (i.e., the variance of $\varphi(0)$) as a function of the parameters $N$, $\beta$ and $\alpha$. For fixed $\alpha$ and $\beta$, we say that the random interface is localised when this variance remains bounded as $N$ tends to infinity, and we say that the interface is delocalised when it diverges as $N$ tends to infinity (N.B. it can be shown that this variance is increasing in $N$ so only one of these two behaviours can occur, see Remark~\ref{rem:remarkmonotony}).

The main result identifies the rate of growth (or absence of growth) of this variance as a function of the length of the chain and is stated below.

\begin{theorem}[Localisation/Delocalisation for the Gaussian long-range chain] \label{thm.gaussian}
    For any inverse temperature $\beta \in (0 , \infty)$ and any range exponent $\alpha \in (1 , \infty)$, there exist two constants $c_\beta := c(\beta , \alpha) > 0$ and $C_\beta := C(\beta, \alpha) < \infty$ such that, for any $N \in \N$,
    \begin{equation*}
        \begin{aligned}
            c_\beta \leq \mathrm{Var}_{N , \beta , \alpha} \left[ \varphi(0) \right] & \leq C_\beta &~~ \mbox{for} ~~& \alpha \in (1 , 2) \\
           c_\beta \ln N \leq \mathrm{Var}_{N , \beta , \alpha} \left[ \varphi(0) \right] & \leq C_\beta \ln N &~~ \mbox{for} ~~& \alpha = 2 ~~\mbox{and}~~ \beta \ll 1\\
            c_\beta N^{\alpha - 2} \leq \mathrm{Var}_{N , \beta , \alpha} \left[ \varphi(0) \right] & \leq C_\beta N^{\alpha - 2} &~~ \mbox{for} ~~& \alpha \in (2 , 3) \\
            c_\beta N/\ln N \leq \mathrm{Var}_{N , \beta , \alpha} \left[ \varphi(0) \right] & \leq C_\beta N / \ln N &~~ \mbox{for} ~~ &\alpha = 3 \\
            c_\beta N \leq \mathrm{Var}_{N , \beta , \alpha} \left[ \varphi(0) \right] & \leq C_\beta N &~~ \mbox{for} ~~ &\alpha >  3 .
        \end{aligned}
    \end{equation*}
\end{theorem}

\begin{remark} \label{rem:remark1.3}
Let us make a few remarks about the previous result:
\begin{itemize}
    \item The above value of the fluctuation exponent in the regime $\alpha\in(1, \infty) \setminus \{ 2 \}$ matches the one conjectured by Fröhlich and Zegarlinski in the early nineties \cite{frohlich1991phase}.
    \item Some of the results of Theorem \ref{thm.gaussian} have been established the articles~\cite{KH82, frohlich1991phase, garban2023invisibility, coquille2024absence} and we refer to Section~\ref{sec:introoldandnew} for a table summarising the various contributions on the localisation/delocalisation of the integer-valued Gaussian chain.
    As all these results can be proved using the same set of tools, i.e., the graph
    surgery techniques introduced in Section~\ref{sec:secprelim}, we include in this article a proof of all of them, previously known or not.
    \item In the case of the range exponent $\alpha = 2$ at low temperature ($\beta \gg 1$), Fr\"{o}hlich and Zegarlinski~\cite{frohlich1991phase} used a Peierls-type argument to prove that the integer-valued Gaussian chain is localised (i.e., the variance of the height is bounded as a function of $N$). This case is not covered by Theorem~\ref{thm.gaussian} and we mention that it would be interesting to find an alternative proof of the localisation of the integer-valued Gaussian chain at low temperature based on graph surgery techniques.
    \item An explicit dependence of the constants in the inverse temperature $\beta$ can be deduced from our proofs. Specifically, it can be shown that there exist two constants $c := c(\alpha) > 0$ and $C := C(\alpha) < \infty$ (depending only on $\alpha$) such that the following lower bounds hold for the constant~$c_\beta$
    \begin{center}
    \begin{tabular}{llllll}
		\hline
			\multicolumn{1}{|l|}{Lower bound on $c_\beta$} & \multicolumn{1}{l|}{$\beta \leq 1$} & \multicolumn{1}{l|}{$\beta \geq 1$} \\ \hline 
            \multicolumn{1}{|l|}{$\alpha \in (1 ,2) $} & \multicolumn{1}{l|}{$c / \beta$} & \multicolumn{1}{l|}{$c e^{- C \beta}$} \\  \hline
            \multicolumn{1}{|l|}{$\alpha = 2$} & \multicolumn{1}{l|}{$c / \beta$} & \multicolumn{1}{l|}{N.A.} \\  \hline
            \multicolumn{1}{|l|}{$\alpha \in (2 , \infty)$} & \multicolumn{1}{l|}{$c / \beta$} & \multicolumn{1}{l|}{$c e^{- C \beta}$} \\  \hline
	\end{tabular}
    \end{center}
    Regarding the constant $C_\beta$, the following upper bounds can be deduced from the proof
    \begin{center}
    \begin{tabular}{llllll}
		\hline
			\multicolumn{1}{|l|}{Upper bound on $C_\beta$} & \multicolumn{1}{l|}{$\beta \leq 1$} & \multicolumn{1}{l|}{$\beta \geq 1$} \\ \hline     
            \multicolumn{1}{|l|}{$\alpha \in (1 ,2)$} & \multicolumn{1}{l|}{$C / \beta$} & \multicolumn{1}{l|}{$C e^{-  c \beta}$} \\  \hline
            \multicolumn{1}{|l|}{$\alpha =  2$} & \multicolumn{1}{l|}{$C / \beta$} & \multicolumn{1}{l|}{N.A.} \\  \hline
            \multicolumn{1}{|l|}{$\alpha \in (2 , 3]$} & \multicolumn{1}{l|}{$C / \beta$} & \multicolumn{1}{l|}{$C / \beta$} \\  \hline
            \multicolumn{1}{|l|}{$\alpha \in (3 , \infty)$} & \multicolumn{1}{l|}{$C / \beta$} & \multicolumn{1}{l|}{$C e^{- c \beta}$} \\  \hline
	\end{tabular}
    \end{center}
        These upper and lower bounds display the correct dependence in the inverse temperature $\beta$ except for the lower bound in the case $\alpha \in (2 , 3]$ and $\beta \gg 1$, where it is conjectured that it should behave like the inverse of $\beta$ (see~\cite[Open question 3]{garban2023invisibility}). We mention that the question of the identification of the correct dependence in the inverse temperature $\beta$ is related to (and in fact strictly simpler than) the question of the \emph{invisibility of the integers} for the integer-valued Gaussian chain in the case $\alpha \in (2 , 3]$ and $\beta \gg 1$ discussed in the paper of Garban~\cite[Open question 3]{garban2023invisibility} (and we refer to this article for much more information and results in this direction).
        \item All the results stated in Theorem~\ref{thm.gaussian} (as well as their proofs) apply to the real-valued Gaussian long-range chain (obtained by assuming that the interfaces are real-valued and sampled according to the distribution whose density with respect to the Lebesgue measure is given by~\eqref{def.IVGFFlongrange}). In this case, the results can be obtained using alternative methods (since the random interface is a multivariate normal distribution, it can be directly studied through its covariance matrix).
\end{itemize}
\end{remark}

\subsubsection{Summary of known and new results for the integer-valued Gaussian chain} \label{sec:introoldandnew}

Below is a table describing the behaviour of the integer-valued Gaussian chain depending on the inverse temperature $\beta$ and the range exponent $\alpha$.
The new results are displayed in bold and the alternative proofs of known results are displayed in blue.

\bigskip

\hspace{-1cm}
\begin{tabular}{llllll}
	\hline
	\multicolumn{1}{|l|}{\textsc{}} & 	\multicolumn{1}{|l|}{\textsc{Lower bounds}}& \multicolumn{1}{l|}{\textsc{Quantitatively}} & \multicolumn{1}{l|}{\textsc{Qualitatively}}  \\ \hline              
	\multicolumn{1}{|l|}{$\alpha=2, \beta \ll 1$} & \multicolumn{1}{|l|}{$c_\beta\log N$ } & \multicolumn{1}{l|}{ Kjaer-Hilhorst~\cite{KH82}, Garban~\cite{garban2023invisibility} + \textbf{\textcolor{blue}{Proposition~\ref{prop:prop4.1}}}} & \multicolumn{1}{l|}{Deloc $\beta \ll1$} \\ \hline
	\multicolumn{1}{|l|}{$\alpha\in(2,3)$} & \multicolumn{1}{|l|}{$c_\beta N^{\alpha-2}$} &\multicolumn{1}{l|}{ Garban~\cite{garban2023invisibility} for $\beta$ small + \textbf{Theorem~\ref{thm.gaussian} for all} $\beta$} & \multicolumn{1}{l|}{ Deloc for all $\beta$~\cite{coquille2024absence}}\\ \hline
	\multicolumn{1}{|l|}{$\alpha=3$} &\multicolumn{1}{|l|}{$c_\beta N/\log N$} & \multicolumn{1}{l|}{ Garban~\cite{garban2023invisibility} for $\beta$ small + \textbf{Theorem~\ref{thm.gaussian} for all} $\beta$} & \multicolumn{1}{l|}{Deloc for all $\beta$~\cite{coquille2024absence}} \\ \hline        
	\multicolumn{1}{|l|}{$\alpha>3$} &\multicolumn{1}{|l|}{$c_\beta N$} & \multicolumn{1}{l|}{ Garban~\cite{garban2023invisibility}  for all $\beta$+ \textbf{\textcolor{blue}{Section~\ref{sec:lowerboundGaussianalpha>3}}}} & \multicolumn{1}{l|}{Deloc for all $\beta$~\cite{coquille2024absence}}  \\ \hline                 
\end{tabular}

\bigskip

\bigskip
\hspace{-1cm}
\begin{tabular}{llllll}
	\hline
	\multicolumn{1}{|l|}{} & \multicolumn{1}{|l|}{\textsc{Upper bounds}}& \multicolumn{1}{l|}{\textsc{Quantitatively}} & \multicolumn{1}{l|}{\textsc{Qualitatively}}  \\ \hline              
	\multicolumn{1}{|l|}{$\alpha\in(1,2)$} & \multicolumn{1}{|l|}{$C_\beta$}&\multicolumn{1}{l|}{ Garban~\cite{garban2023invisibility} for all $\beta$ + \textbf{\textcolor{blue}{Section~\ref{sec.Baumlertechnique}}}} & \multicolumn{1}{l|}{Localised for all $\beta$} \\ \hline
	\multicolumn{1}{|l|}{$\alpha=2, \beta \gg 1$} &\multicolumn{1}{|l|}{$C_\beta$}& \multicolumn{1}{l|}{ Fr\"{o}hlich-Zegarlinski~\cite{frohlich1991phase}}& \multicolumn{1}{l|}{Localised $\beta\gg1$} \\ \hline
	\multicolumn{1}{|l|}{$\alpha=2, \beta \ll 1$} & \multicolumn{1}{|l|}{$C_\beta\log N$}&\multicolumn{1}{l|}{  Garban~\cite{garban2023invisibility} + \textbf{\textcolor{blue}{Section~\ref{sec.Baumlertechnique}}}}& \multicolumn{1}{l|}{Delocalised  $\beta\ll1$} \\ \hline
	\multicolumn{1}{|l|}{$\alpha \in (2,3)$} &\multicolumn{1}{|l|}{$C_\beta N^{\alpha-2}$}& \multicolumn{1}{l|}{ Garban~\cite{garban2023invisibility} for all $\beta$ + \textbf{\textcolor{blue}{Section~\ref{sec.Baumlertechnique}}}} & \multicolumn{1}{l|}{Delocalised for all $\beta$}\\ \hline 
	\multicolumn{1}{|l|}{$\alpha =3$} & \multicolumn{1}{|l|}{$C_\beta N/\log N$}&\multicolumn{1}{l|}{\textbf{Section~\ref{sec:casealpha=3}}} & \multicolumn{1}{l|}{Delocalised for all $\beta$}\\ \hline          
	\multicolumn{1}{|l|}{$\alpha > 3$} &\multicolumn{1}{|l|}{$C_\beta N$}& \multicolumn{1}{l|}{ Garban~\cite{garban2023invisibility} for all $\beta$ + \textbf{\textcolor{blue}{Section~\ref{sec:casealpha>3}}}} & \multicolumn{1}{l|}{Delocalised for all $\beta$}\\ \hline    
\end{tabular}

\subsubsection{Discrete $q$-SOS long-range chain} \label{sec:introqSOSC}

In order to go beyond the integer-valued Gaussian chain, it is interesting to study other potentials than the square function in Definition~\ref{def.def1.1}. In Section~\ref{sec.section4qSOS}, we consider the $q$-SOS long-range chain with $q \in (0,2)$ (which is obtained by replacing the square function in~\eqref{def.IVGFFlongrange} by the function $x \mapsto |x|^q$) and show that the fluctuations of this model are different from that of the Gaussian chain. The precise definition of the model and the results obtained in Section~\ref{sec.section4qSOS} are stated below, and summarised on Figure \ref{fig:alpha-q-plane}.

\begin{definition}[Finite-volume discrete $q$-SOS long-range chain]
For $N \in \N$, $q \in (0,2)$, $\beta \in (0 , \infty)$ and $\alpha \in (1 , \infty)$, we define the discrete $q$-SOS long-range chain of length $N$ at inverse temperature $\beta$ and with range exponent $\alpha$ to be the probability distribution on the set $\Omega_N$ given be the identity, for any $\varphi \in \Omega_N,$
\begin{equation*}
    \mu_{N , \beta , \alpha}^{q-\mathrm{SOS}}(\{ \varphi \})  = \frac{1}{Z_{N , \beta , \alpha}^{q-\mathrm{SOS}}} \exp \left( - \beta \sum_{i , j \in \Z} \frac{\left| \varphi(i) - \varphi(j) \right|^q}{|i - j|^\alpha} \right),
\end{equation*}
where $Z_{N , \beta , \alpha}^{q-\mathrm{SOS}}$ is the normalizing constant. We denote by $\mathrm{Var}_{N , \beta , \alpha}^{q-\mathrm{SOS}}$ the variance with respect to $\mu_{N , \beta , \alpha}^{q-\mathrm{SOS}}$.
\end{definition}

\begin{theorem}[Delocalisation for the $q$-SOS long-range chain] \label{thm.qSOS}
	For any $\beta \in (0 , \infty)$, any range exponent $\alpha \in (1 , \infty)$, and any $q\in(0,2)$ there exists a constant $c_\beta := c(\beta , \alpha, q) > 0$ such that, for any $N \in \N$,
	\begin{equation*}
	\begin{aligned}
    c_\beta  \leq \mathrm{Var}_{N , \beta , \alpha}^{q-\mathrm{SOS}} \left[ \varphi(0) \right]  ~~~  &~~ \mbox{for} ~& \alpha &\in \left( 1, 2 \right], \\
	c_\beta N^{\frac{2}{q}(\alpha - 2)} \leq \mathrm{Var}_{N , \beta , \alpha}^{q-\mathrm{SOS}} \left[ \varphi(0) \right] ~~~  &~~ \mbox{for} ~& \alpha &\in \left( 2 , 2 + \frac{q}{2} \right), \\
	 c_\beta \frac{N}{(\ln N)^{2/q}}\leq \mathrm{Var}_{N , \beta , \alpha}^{q-\mathrm{SOS}} \left[ \varphi(0) \right] ~~~ &~~ \mbox{for} ~ &\alpha &= 2 + \frac{q}{2}, \\
	 c_\beta N \leq \mathrm{Var}_{N , \beta , \alpha}^{q-\mathrm{SOS}} \left[ \varphi(0) \right] ~~~ &~~ \mbox{for} ~ &\alpha &> 2 + \frac{q}{2} .
	\end{aligned}
	\end{equation*}
    Regarding the upper bound, there exists a constant $C_\beta := C(\beta , \alpha, q) < \infty$ such that
    \begin{equation*}
        \mathrm{Var}_{N , \beta , \alpha}^{q-\mathrm{SOS}} \left[ \varphi(0) \right] \leq \left\{ \begin{aligned}
            C_\beta & ~\mbox{for} ~ \alpha < q, \\
            C_\beta \ln N & ~\mbox{for} ~ \alpha = q, \\
            C_\beta N^{\frac{2 \alpha}{q} - 2} & ~\mbox{for} ~ \alpha > q.
        \end{aligned} \right.
    \end{equation*}
\end{theorem}

\begin{remark}
Let us make a few remarks about the previous result:
\begin{itemize}
	\item Delocalisation in the qualitative sense (absence of shift-invariant Gibbs measure) under a $q$-th moment condition can be derived for any $\alpha>2$ and any $q > 2$ in the same fashion as in the paper~\cite{coquille2024absence}. We thus add this weak delocalisation zone on Figure \ref{fig:alpha-q-plane}.
    \item Unfortunately, the method does not seem to yield matching upper and lower bounds for the fluctuations of the chain (the reason is technical and discussed in Remark~\ref{rem:rem4.6} below). We believe that it is an interesting open question to identify the correct fluctuation exponent. Considering our proofs in more details, there seems to be more space for improvement in the upper bound case, and it is thus tempting to conjecture that the lower bound should be the correct one.
    \item As it was the case for the Gaussian chain, an explicit dependence of the constants in the inverse temperature $\beta$ could be extracted from the argument (but this dependence is more intricate to follow than in the Gaussian case, we thus decided not to keep track of it in Section~\ref{sec.section4qSOS}).
    \item All the results stated in Theorem~\ref{thm.qSOS} and their proofs apply to the real-valued $q$-SOS long-range chain (N.B. for this model, the distribution of the real-valued random interface cannot be identified as easily as in the Gaussian case).
    \item When $q \to \infty$, the interface $\varphi$ will take
    values either in $\{0,1\}$ or in $\{0,-1\}$ with probability $1/2$, so the limit $q \to \infty$ corresponds to a symmetric convex combination of Ising models.
\end{itemize}
\end{remark}

\begin{figure}
	\begin{center}
		\includegraphics[scale=0.55]{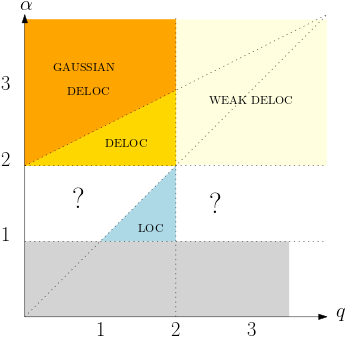}
		\caption{Partial phase diagram of the $q$-SOS long-range chain.} \label{fig:alpha-q-plane}
	\end{center}
\end{figure}

\subsection{Related results}

A few (already mentioned) results have been established on the integer-valued Gaussian chain, and we now present a few more details about the statements and techniques of proof.

\begin{itemize}
	\item In~\cite{KH82}, Kjaer and Hilhorst study the integer-valued Gaussian chain (with periodic boundary conditions) and use a duality transformation to derive an exact identity for the variance of the height of the chain when $\beta = 1$ and when the long-range interaction is exactly $\frac{1}{|i - j|^2 - 1/4}$. Their identity shows in particular that this variance diverges logarithmically fast in the length of the chain. We mention that this approach was recently extended to the two-dimensional setting by Cornu, Hilhorst and Bauer~\cite{cornu2023new}.
	\item In~\cite{frohlich1981kosterlitz}, Fr\"{o}hlich and Zegarlinski study the behaviour of the chain in the regime $\alpha = 2$ and $\beta \gg 1$ and show that the integer-valued Gaussian chain is localised. The proof relies on the implementation of a Peierls-type argument (inspired by the one developed in the work of Fr\"{o}hlich and Spencer~\cite{frohlich1982phase} for the 1D long-range Ising model). Combined with~\cite{KH82} (and a correlation inequality), their result implies the existence of a phase transition for the Gaussian chain for $\alpha = 2$ between a localised regime at low temperature and a delocalised regime at high temperature.
	\item In~\cite{garban2023invisibility}, Garban studies the integer-valued Gaussian chain in the high temperature regime (i.e., $\beta \ll 1$) for the range exponent $\alpha \in [2 , 3)$ and obtains (among other results) very precise information on the behaviour of the chain by identifying its scaling limit. His result exhibits an interesting phenomenon called \emph{the invisibility of the integers}: the scaling limit of the integer-valued Gaussian chain is shown to be exactly the same as the one of the real-valued Gaussian chain (N.B. this property is known to be false for related models such as the integer-valued Gaussian chain with $\alpha > 3$ and $\beta \gg 1$~\cite[Remark 14]{garban2023invisibility}, or for the two-dimensional integer-valued Gaussian free field~\cite{GS23}). Let us emphasize that the results of~\cite{garban2023invisibility} include the case $\alpha = 2$ (for small $\beta$) even though the Gaussian chain is known to exhibit a phase transition in this regime. On a currently less rigorous level, it is conjectured that the invisibility of integers holds in the case $\alpha = 2$ for any $\beta$ (strictly) smaller than the critical threshold for the localisation/delocalisation phase transition (see~\cite{slurink1983roughening} and~\cite[Open Problem 4]{garban2023invisibility}) and that it holds for any value of $\beta \in (0 , \infty)$ in the case $\alpha \in (2 , 3)$ (see~\cite[Open Problem 3]{garban2023invisibility}).
    
    Other results obtained in~\cite[Theorem 1.1 and Proposition 1.4]{garban2023invisibility} include upper bounds for the fluctuations of the interface at every inverse temperature and for any $\alpha > 1$ and lower bounds in the high temperature regime for any $\alpha \geq 2$ (the high temperature constraint can be removed for $\alpha > 3$).
	\item In~\cite{coquille2024absence}, two of the authors of the present paper showed with van Enter and Ruszel the absence of shift-invariant Gibbs states (i.e. a qualitative statement of delocalisation) of the $q$-SOS chain at every inverse temperature for any $q\in[1,2]$ and range exponent $\alpha > 2$ by adapting an argument for the 1D long-range Ising model (see~\cite{coquille2018absence} and references therein). The absence of shift-invariant Gibbs state $ \mu_{\beta , \alpha}^{q-\mathrm{SOS}}$ under a $q$-th moment condition ($ \mu_{\beta , \alpha}^{q-\mathrm{SOS}}(|\varphi(0)|^q)<\infty$) can be derived for any $\alpha>2$ and any $q > 2$ in the same fashion as in the paper \cite{coquille2024absence}.
\end{itemize}

The localisation/delocalisation phase transition of the integer-valued Gaussian chain for $\alpha = 2$ is closely related to the localisation/delocalisation phase transition of the 2D integer-valued Gaussian free field (see Definition~\ref{def:2DGFF} below). The existence of this phase transition was first established in the seminal article of Fr\"{o}hlich and Spencer~\cite{frohlich1981kosterlitz}, and a new proof has been recently obtained by Lammers~\cite{lammers2022height}. Since then, there has been an important number of recent remarkable developments on the topic~\cite{kharash2017fr, wirth2019maximum, garban2023statistical, GS23, BRP1, BRP2, P24, LO24, lammers2022dichotomy, aizenman2021depinning, van2023elementary, van2023duality, laslier2024tilted, biskup2024}.

\subsection{Strategy of the proof} \label{section:stratofproof}

As mentioned above the strategy of the proof of Theorem~\ref{thm.gaussian} relies on graph surgery. To explain the argument in more details, we first introduce a definition.

\begin{definition}\label{def:IV-GFF}
	To each finite connected rooted graph $G := (V , E , x)$ (with vertex set $V$, edge set $E$, root $x \in V$) equipped with conductances $\lambda := (\lambda_{ij})_{ij \in E} \in (0 , \infty)^E$, we associate a random interface by equipping the set of interfaces (or height functions) $\Omega_V := \left\{ \varphi : V \to \Z \, : \, \varphi(x) = 0 \right\}$ with the probability distribution
	\begin{equation*}
	\mathbb{P}^{\mathrm{IV-GFF}}_{G , \lambda}( \{ \varphi \}) := \frac{1}{Z_{G , \lambda}} \exp \left( - \sum_{ij \in E} \lambda_{ij} \left( \varphi(i) - \varphi(j) \right)^2 \right).
	\end{equation*}
	We denote by $\mathrm{Var}_{G , \lambda}^{\mathrm{IV-GFF}}$ the variance with respect to the measure $\mathbb{P}^{\mathrm{IV-GFF}}_{G , \lambda}$.
\end{definition}
In this formalism, the integer-valued Gaussian long-range chain introduced in Definition~\ref{def.def1.1} can be represented as the line $\{ - N , \ldots , N \} \subseteq \Z$ with long-range edges connecting any two pair of vertices $i , j \in \{ - N , \ldots , N \}$ and the conductance of the edge $ij$ is equal to $\beta / |i - j|^{\alpha}$ (N.B. this is not fully correct as stated and we refer to Remark~\ref{rem:partitionfunctionmonotone} for the exact interpretation). This graph will be denoted by $G_N$ in the rest of this section (see Figure~\ref{fig:sectionstrategy}).

\begin{figure} 
\begin{center}
\includegraphics[scale=0.45]{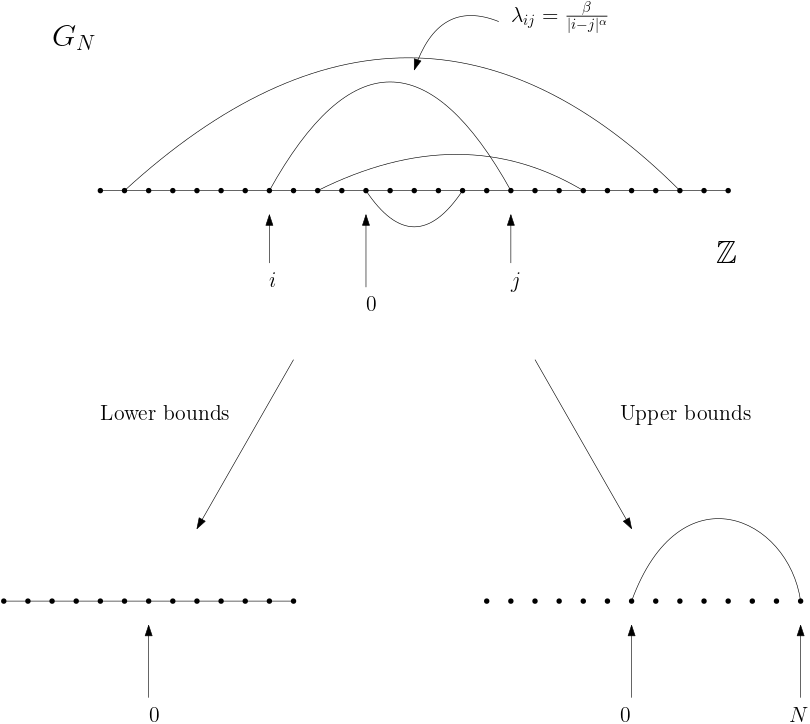}
\caption{Strategy of the proof of Theorem~\ref{thm.gaussian}. The graph $G_N$ associated with the integer-valued Gaussian chain~\eqref{def.IVGFFlongrange} is drawn on the top line. This graph is then simplified using the operations (i), (ii) and (iii) so as to either obtain a nearest neighbour line for the lower bounds or to obtain a graph containing only one edge for the upper bound. On the resulting graphs, the variance of $\varphi(0)$ can be easily computed.} \label{fig:sectionstrategy}
\end{center}
\end{figure}

Let us then fix a finite rooted graph $G := (V , E , x)$ equipped with conductances as well as a vertex $v \in V$ (N.B. this graph can be any graph and is not necessarily the graph $G_N$). The core of the proof of Theorem~\ref{thm.gaussian} relies on the fact that the following operations on the graph $G$ have a monotonic effect on the variance of the height $\varphi(v)$:
\begin{enumerate}
    \item[(i)] Erasing an edge from the graph $G$ as in Figure~\ref{removingedge.picture} increases the variance of the height $\varphi(v)$ (see Proposition~\ref{corollary.deletion.edges});
    \item[(ii)] Identifying two vertices of the graph $G$ as in Figure~\ref{identification.picture} reduces the variance of $\varphi (v)$ (see Proposition~\ref{corollary.identification});
    \item[(iii)] Adding a new vertex in the middle of an edge while suitably adjusting the conductances as in Figure~\ref{adition1.picture} reduces the variance of $\varphi(v)$ (see Proposition~\ref{prop.prop2.9}).
\end{enumerate}
These operations can then be combined so as to simplify the graph $G_N$ associated with the integer-valued Gaussian chain~\eqref{def.IVGFFlongrange}. More specifically, we will proceed differently for the lower and upper bounds (see Figure~\ref{fig:sectionstrategy}):
\begin{itemize}
    \item The lower bounds: in this case, we will use the operations (i), (ii) and (iii) above to reduce the graph $G_N$ to a nearest neighbour line while reducing the variance of the height $\varphi(0)$.
    \item The upper bounds: in this case, we adapt an argument of Ba\"{u}mler~\cite[Proof of Theorem 1.1]{baumler2023recurrence} and use the operations (i) and (ii) above to reduce the graph $G_N$ to the graph whose vertex set is the line $\{ - N , \ldots , N \}$ but which contains only one edge connecting the vertex $0$ to the vertex $N$ (N.B. we have the constraint $\varphi(N) = 0$) while increasing the variance of the height $\varphi(0)$.
\end{itemize}
In both cases, the variance of $\varphi(0)$ on the resulting graph can be easily computed, yielding the estimates of Theorem~\ref{thm.gaussian}.

The proof of Theorem~\ref{thm.qSOS} for the $q$-SOS long-range chain is based on a similar strategy with one additional idea due to Sellke~\cite{sellke2024localization}. We first rely on the property that the power-law interaction $x \mapsto e^{-|x|^q}$ (with $q \in (0 , 2]$) can be decomposed into a mixture of centred Gaussian densities, i.e., there exists a probability measure $\mu_q$ on $[0 , \infty)$ such that, for any $x \in \R$,
\begin{equation} \label{eq:stratmixturegaussian}
    e^{- |x|^q} = \int_{0}^\infty e^{- \lambda |x|^2} \mu_q(d \lambda).
\end{equation}
The identity~\eqref{eq:stratmixturegaussian} can be used to show that the variance of the height $\varphi(0)$ for the $q$-SOS long-range chain is equal to the variance of the height $\varphi(0)$ of an integer-valued Gaussian chain equipped with random conductances (N.B. this observation has already been used to study random interfaces and a brief account of the literature can be found in Section~\ref{sec:section2.5.1}). This latter model can be analysed using the same techniques as the ones used in the proof of Theorem~\ref{thm.gaussian} together with the FKG correlation inequality (N.B. the use of the FKG inequality is one of the novelties of~\cite{sellke2024localization}).

\bigskip

\section{Definitions and preliminaries} \label{sec:secprelim}

In this section, we introduce some notation and preliminary results. Specifically, we introduce a general version of the integer-valued Gaussian free field and state some monotonicity results (monotonicity in the conductances, under addition and identification of vertices and under deletion of edges) which are then used to perform the graph surgery in the proofs of Theorems~\ref{thm.gaussian} and~\ref{thm.qSOS}. We remark that these monotonicity properties have been recently used by Aizenman, Harel, Peled and Shapiro~\cite{aizenman2021depinning} and by van Engelenburg and Lis~\cite{van2023elementary, van2023duality} (to establish either the existence of a phase transition for two-dimensional integer-valued height functions or the existence of a Berezinskii-Kosterlitz-Thouless -- BKT -- phase transition for $O(1)$ spin systems) and by Garban~\cite{garban2023invisibility} to study the Gaussian chain.

\subsection{General notation}

In all this section, we consider a (finite) rooted graph $G = (V , E , x)$ where $V$ is the set of vertices, $E$ is the set of edges of $G$ and $x \in V$ is the root. 

The graph $G$ can then be equipped with conductances by considering a collection of non-negative real numbers $\left( \lambda_{ij} \right)_{ij \in E} \in [0 , \infty)^E$ indexed by the edges of $G$ (or equivalently a function from $E$ to $(0 , \infty)$). 

We equip the collection of conductances of the graph $G$ with a partial order as follows: given two collections of conductances $\lambda = \left( \lambda_{ij} \right)_{ij \in E} \in (0 , \infty)^E $ and $\lambda' = \left( \lambda'_{ij} \right)_{ij \in E} \in [0 , \infty)^E $, we write
\begin{equation*}
    \lambda \succeq \lambda' ~~ \iff ~~ \forall ij \in E, ~ \lambda_{ij} \geq \lambda'_{ij}.
\end{equation*}
A function $f : [0 , \infty)^E \to \R$ is called increasing (resp. decreasing) if
\begin{equation} \label{def.increasingfunction}
    \forall \lambda, \lambda' \in [0 , \infty)^E, ~ \lambda \succeq \lambda' ~ \implies ~ f(\lambda) \geq f(\lambda') ~~~(\mbox{resp.}~ f(\lambda) \leq f(\lambda')).
\end{equation}
We denote by $\Omega_V := \left\{ \varphi : V \mapsto \Z \, : \, \varphi(x) = 0 \right\} \simeq \Z^{V \setminus \{ x \}}$ the set of integer-valued height functions on the rooted graph $G = (V , E , x)$.

\subsection{Integer-valued Gaussian free field on a finite graph}

In this subsection, we introduce the integer-valued Gaussian distribution and the integer-valued Gaussian free field on the graph $G$.

\subsubsection{integer-valued Gaussian random variable}

\begin{definition}
    For $\lambda > 0$, we define the integer-valued Gaussian distribution of conductance $\lambda$ to be the probability distribution on $\Z$ given by the identity: for any $k \in \Z$,
    \begin{equation*}
        \mathbb{P}^{\mathrm{IV-G}}_\lambda (\{ k \}) := \frac{e^{- \lambda k^2}}{\sum_{k \in \Z} e^{- \lambda k^2}}.
    \end{equation*}
    We denote by $\mathrm{Var}^{\mathrm{IV-G}}_\lambda$ the variance with respect to $\mathbb{P}^{\mathrm{IV-G}}_\lambda$.
\end{definition}

A random variable whose law is an integer-valued Gaussian distribution is called an integer-valued Gaussian random variable (of conductance $\lambda$).

In the following lemma, we let $X$ be an integer-valued Gaussian random variable of conductance $\lambda$ and estimate its variance as a function of $\lambda$. 

\begin{lemma} \label{lemma:lemma2.2}
    There exist two constants $c > 0$ and $C < \infty$ such that:
    \begin{itemize}
        \item Gaussian domination bound~\cite{frohlich1978correlation} (cf.~\cite{kharash2017fr}): for any $\lambda \in (0 , \infty),$
        \begin{equation*}
            \mathrm{Var}^{\mathrm{IV-G}}_\lambda [X] \leq \frac{1}{2\lambda}.
        \end{equation*}
        \item Behaviour as $\lambda \to 0$: for any $\lambda \in (0 , 1)$,
        \begin{equation*}
            \frac{c}{\lambda} \leq \mathrm{Var}^{\mathrm{IV-G}}_\lambda [X] \leq \frac{1}{2\lambda}.
        \end{equation*}
        \item Behaviour as $\lambda \to \infty$: for any $\lambda \in (1 , \infty)$,
        \begin{equation*}
              c e^{- \lambda}  \leq \mathrm{Var}^{\mathrm{IV-G}}_\lambda [X] \leq C e^{- \lambda}.
        \end{equation*}
    \end{itemize}
\end{lemma}

\subsubsection{Integer-valued Gaussian free field on a finite graph}

We recall the Definition~\ref{def:IV-GFF} of the integer-valued GFF and precise that $Z_{G , \lambda}$ is the normalizing constant (or partition function) defined by the identity
\begin{equation} \label{def.partitionfunctioncond}
    Z_{G , \lambda} := \sum_{\varphi \in \Omega_V} \exp \left( - \sum_{ij \in E} \lambda_{ij} \left( \varphi(i) - \varphi(j) \right)^2 \right).
\end{equation}

\begin{remark}\label{rem:partitionfunctionmonotone}
Let us make a few remarks:
\begin{itemize}
    \item Definition~\ref{def:IV-GFF} only makes sense if the graph $G$ is connected and we may have to consider the variance of the height of a vertex $v \in V$, i.e., $\mathrm{Var}_{G , \lambda}^{\mathrm{IV-GFF}}[\varphi(v)] $, for disconnected graphs. In order to do so, we extend the definition of this quantity to disconnected graphs by using Definition~\ref{def:IV-GFF} on the connected component of the root and by setting $\mathrm{Var}_{G , \lambda}^{\mathrm{IV-GFF}}[\varphi(v)] = \infty$ for any vertex $v$ which is not in this connected component. 
    \item The model~\eqref{def.IVGFFlongrange} is obtained by considering the complete graph with $2N$ vertices, identifying the set of vertices with the set $\{ - N +1 , \ldots , N \}$ (in an arbitrary way) and then considering the conductances 
    $$\lambda_{ij} = \frac{\beta}{|i - j|^\alpha} ~ \mbox{for} ~ i , j \in \{ - N +1 , \ldots , N -1 \}$$ 
    and $$\lambda_{iN} := \sum_{j \in \Z \setminus \{ - N +1 , \ldots , N - 1 \}} \frac{\beta}{|i - j|^\alpha}  ~ \mbox{for} ~ i \in \{ - N +1 , \ldots , N -1 \}.$$
    We denote by $ \mathrm{Var}_{N,\beta,\alpha}$ the variance under the corresponding probability measure.
    \item A simple graph of interest for us is the nearest neighbour chain defined as follows. Given an integer $L \in \N$, we consider the graph whose vertices are $\{ 0 , \ldots, L\}$, equipped with nearest neighbour edges and rooted in $L$. We then consider a collection of conductances $\lambda := (\lambda_{i(i+1)})_{i \in \{0 , \ldots, L-1\}}$ and let $\varphi : \{ 0 , \ldots , L\} \to \Z$ be sampled according to the integer-valued Gaussian free field on this graph. The variance of the height $\varphi(0)$, denoted by $\mathrm{Var}_{L , \lambda}\left[ \varphi(0) \right]$, can then be explicitly computed by writing 
    $$\varphi(0) = \sum_{i = 0}^{L-1} \left( \varphi(i) - \varphi(i + 1)\right)$$
    and by observing that the random variables $\lambda :=  \left( \varphi(i) - \varphi(i + 1)\right)_{i \in \{ 0 , \ldots , L - 1\}}$ are independent and that they follow the integer-valued Gaussian distribution with conductance $(\lambda_{i(i+1)})_{i \in \{0 , \ldots, L-1 \} }$. In particular, we have
    \begin{equation*}
        \mathrm{Var}_{L , \lambda}\left[ \varphi(0) \right] = \sum_{i = 0}^{L-1} c_{\lambda_{i(i+1)}},
    \end{equation*}
    where $c_{\lambda_{i(i+1)}}$ denotes the variance of an integer-valued Gaussian random variable of conductance $\lambda_{i(i+1)}$. In the case where all the conductances are equal to the same value $\lambda \in (0, \infty),$ we have
    \begin{equation} \label{eq:varianceforthechain}
        \mathrm{Var}_{L , \lambda}\left[ \varphi(0) \right] = c_{\lambda} L.
    \end{equation}
    \item The function $\lambda \mapsto Z_{G , \lambda}$ is decreasing in $\lambda$ (as increasing the values of the conductances reduces each term of the sum in the right-hand side of~\eqref{def.partitionfunctioncond}).
\end{itemize}
\end{remark}

\subsection{Monotonicity of the variance of the height in the conductances} \label{sectionconductancemonotone}

In this section, we state a monotonicity property in the conductances for the height of the integer-valued Gaussian free field and collect two corollaries (the monotonicity under deletion of vertices and identification of vertices). These results are all an almost immediate consequence of the Regev-Stephens-Davidowitz monotonicity theory~\cite{regev2017inequality} and they will be used extensively in the proofs below. 

\subsubsection{Monotonicity in the conductances}

The following statement asserts that the variance of the height and the differences of the heights of an integer-valued Gaussian free field are decreasing functions of the conductances.

\begin{proposition}[Monotonicity of the height and differences of heights in the conductances~\cite{regev2017inequality}] \label{prop:monotonicityconductances}
The following statement holds:
\begin{equation*}
\forall (i , j) \in  V \times V, ~ \mbox{the function}~ \lambda \mapsto \mathrm{Var}_{G , \lambda}^{\mathrm{IV-GFF}}\left[ \varphi(i) - \varphi(j)  \right] ~\mbox{is decreasing in} ~ \lambda \in (0 , \infty)^E.
\end{equation*}
In particular, by setting $j = x$,
\begin{equation*}
\forall i \in  V, ~ \mbox{the function}~ \lambda \mapsto \mathrm{Var}_{G , \lambda}^{\mathrm{IV-GFF}}\left[ \varphi(i)  \right] ~\mbox{is decreasing in} ~ \lambda \in (0 , \infty)^E.
\end{equation*}
\end{proposition}

\begin{remark}
More generally, these properties are valid for linear functionals of the height function $\varphi$ with the moment generating function replacing the variance, see~\cite{regev2017inequality}.
\end{remark}

We collect in the following subsections two corollaries of Proposition~\ref{prop:monotonicityconductances}: the monotonicity of the variance of the height under deletion of edges and under identification of vertices. They are obtained by sending the value of a specific conductance $\lambda_{ij}$ to the extremal values $0$ and infinity in Proposition~\ref{prop:monotonicityconductances}.

\subsubsection{Monotonicity under deletion of edges}

\begin{proposition}[Monotonicity under deletion of edges] \label{corollary.deletion.edges}
Let $yz \in E$ be an edge of $G$ and $\lambda \in (0 , \infty)^E$ be a collection of conductances. Denote by $G^*$ the graph $G$ from which the edge $yz$ has been removed and by $\lambda^*$ the restriction of $\lambda$ to the edges of $G^*$ (see Figure~\ref{removingedge.picture}). Then
\begin{equation*} 
     \forall ~ (i , j) \in V \times V,~\mathrm{Var}_{G , \lambda}^{\mathrm{IV-GFF}}\left[  \varphi(i) - \varphi(j) \right] \leq \mathrm{Var}_{G^* , \lambda^*}^{\mathrm{IV-GFF}}\left[ \varphi(i) - \varphi(j) \right].
\end{equation*}
In particular, by setting $j = x$,
\begin{equation*}
    \forall ~ i \in  V,~\mathrm{Var}_{G , \lambda}^{\mathrm{IV-GFF}}\left[ \varphi(i)  \right] \leq \mathrm{Var}_{G^* , \lambda^*}^{\mathrm{IV-GFF}}\left[ \varphi(i) \right].
\end{equation*}

\end{proposition}

\begin{figure} 
\begin{center}
\includegraphics[width=10cm]{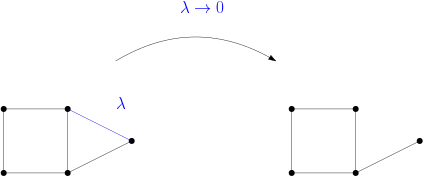}
\caption{Removing an edge from a graph. This operation increases the variance of the linear functionals of the field} \label{removingedge.picture}
\end{center}
\end{figure}

\begin{proof}
    This result is obtained by considering the graph $G$ with conductances $\lambda \in (0 , \infty)^E$ and reducing the value of the conductance $\lambda_{yz}$ to $0$. This operation removes the edge $yz$ from the graph and increases the variances of $\varphi(i)$ and $\varphi(i) - \varphi(j)$ by Proposition~\ref{prop:monotonicityconductances}.
\end{proof}

\subsubsection{Monotonicity under identification of vertices} \label{sec:monotneconductance}

We now state the monotonicity under the identification of vertices. For the definition of $\bar \lambda$ below, we will use the convention $\lambda_{xy} = 0$ if $xy \notin E$.

\begin{proposition}[Monotonicity under identification of vertices] \label{corollary.identification}
Let $y_1 , y_2 \in V$ be two vertices of $G$ and $\lambda \in (0 , \infty)^E$ be a collection of conductances. Denote by $\bar G = (\bar V, \bar E, x)$ the graph $G$ in which the vertices $y_1$ and $y_2$ have been identified, denote by $y$ this new vertex, and let $\bar \lambda : \bar E \to (0 , \infty)$ the collection of conductances defined by (see Figure~\ref{identification.picture})
\begin{equation*}
     \bar \lambda_{y z} := \lambda_{y_1 z} + \lambda_{y_1 z} ~ \mbox{for}~ yz \in \bar E ~\mbox{and} ~ \bar \lambda_{z z} = \lambda_{z z'}  ~ \mbox{for}~ zz' \in \bar E  ~\mbox{with}~ z \neq y ~\mbox{and} ~ z' \neq y . 
\end{equation*}
Then
\begin{equation*}
         \forall ~ (i , j) \in V \times V,~\mathrm{Var}_{\bar G , \bar \lambda}^{\mathrm{IV-GFF}}\left[  \varphi(i) - \varphi(j)  \right] \leq \mathrm{Var}_{G , \lambda}^{\mathrm{IV-GFF}}\left[  \varphi(i) - \varphi(j)  \right].
\end{equation*}
In particular, by setting $j = x$,
\begin{equation*} 
      \forall ~ i \in  V,~\mathrm{Var}_{\bar G , \bar \lambda}^{\mathrm{IV-GFF}}\left[  \varphi(i) \right] \leq \mathrm{Var}_{G , \lambda}^{\mathrm{IV-GFF}}\left[ \varphi(i) \right].
\end{equation*}
\end{proposition}

\begin{figure} 
\begin{center}
\includegraphics[width=10cm]{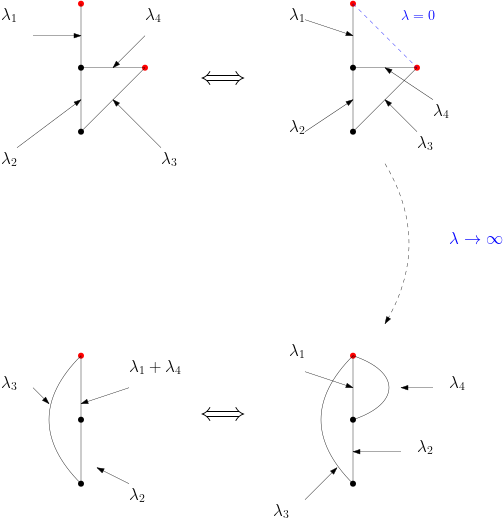}
\caption{Identifying the vertices in a graph. This operation reduces the variance of linear functionals of the field} \label{identification.picture}
\end{center}
\end{figure}

\begin{proof}
    We distinguish two cases: whether $yz$ is an edge of $E$ or not (i.e., $yz \in E$ or $yz \notin E$).
    
    In the first case, we increase the value of the conductance $\lambda_{yz}$ to infinity. This operation identifies the vertices $y$ and $z$ and reduces the variances of $\varphi(i)$ and $\varphi(i) - \varphi(j)$ by Proposition~\ref{prop:monotonicityconductances}.

    In the second case, we let $G'$ be the graph $G$ to which the edge $yz$ has been added and extend the collection of conductances $\lambda$ to the edge set of $G'$ by setting $\lambda_{yz} = 0$. We then increase the value of the conductance $\lambda_{yz}$ from $0$ to infinity.
\end{proof}

\begin{remark} \label{rem:remarkmonotony}
    We conclude this section by mentioning three monotonicity properties which are direct consequences of Proposition~\ref{prop:monotonicityconductances} and Proposition~\ref{corollary.identification}. The variance $\mathrm{Var}_{N , \beta , \alpha}\left[ \varphi(0) \right]$ is:
\begin{itemize}
    \item Increasing in the length of the chain $N$ (this is obtained by using the second item of Remark~\ref{rem:partitionfunctionmonotone} and by applying Proposition~\ref{corollary.identification} to identify the vertices $N-1$ and $-N+1$ with the vertex $N$);
    \item Decreasing in the inverse temperature $\beta$ (this is a consequence of Proposition~\ref{prop:monotonicityconductances});
    \item Increasing in the range exponent $\alpha$ (this is a consequence of Proposition~\ref{prop:monotonicityconductances}).
\end{itemize}
\end{remark}

\subsection{Monotonicity of the variance of the height under addition of vertices}

In this section, we state a second monotonicity property which allows to add vertices to a graph while reducing the variances of the heights and differences of heights upon suitable modification of the conductances (see Figure~\ref{adition1.picture}). We note that this result was used by Aizenman, Harel, Peled and Shapiro~\cite{aizenman2021depinning} to study the phase transition of the two-dimensional integer-valued Gaussian free field.

\begin{proposition}[Monotonicity under addition of vertices] \label{prop.prop2.9}
Given an edge $yz \in E$, we denote by $\tilde G$ the graph obtained by adding a new vertex $t$ on the edge $yz$ (see Figure~\ref{adition1.picture}). We let $\tilde \lambda$ be a collection of conductances on the graph $\tilde G$ satisfying:
\begin{itemize}
\item On all the edges of $\tilde G$ which are not incident to $t$, $\tilde \lambda = \lambda$,
\item For the edges $yt$ and $tz$, $\frac{1}{\tilde \lambda_{yt}} + \frac{1}{\tilde \lambda_{tz}} = \frac{1}{\lambda_{yz}}$.
\end{itemize}
Then
\begin{equation*}
         \forall ~ (i , j) \in V \times V,~\mathrm{Var}_{\tilde G , \tilde \lambda}^{\mathrm{IV-GFF}}\left[\varphi(i) - \varphi(j) \right] \leq \mathrm{Var}_{G , \lambda}^{\mathrm{IV-GFF}}\left[ \varphi(i) - \varphi(j)  \right] .
\end{equation*}
In particular, by setting $j = x$,
\begin{equation*} 
     \forall ~ i \in  V,~\mathrm{Var}_{\tilde G , \tilde \lambda}^{\mathrm{IV-GFF}}\left[ \varphi(i) \right] \leq \mathrm{Var}_{G , \lambda}^{\mathrm{IV-GFF}}\left[\varphi(i) \right].
\end{equation*}
\end{proposition}

We briefly sketch the proof of this inequality and refer to~\cite[Section 2, Proof of Theorem 1.2]{aizenman2021depinning} for the formal argument:
\begin{itemize}
    \item The first step is to use the divisibility of the Gaussian distribution to show that one can add the vertex $t$ on the edge $yz$ (while modifying the conductances as in the statement of the proposition) without modifying the distribution of the height function as long as the height $\varphi(t)$ is assumed to be \emph{real-valued}. 
    \item The second step of the argument is to use a correlation inequality to show that conditioning the height $\varphi(t)$ to take integer values reduces the variances of the height $\varphi(i)$ and of the difference $\varphi(i) - \varphi(j)$. This second step can be proved using two different techniques: either the sublattice monotonicity of~\cite{regev2017inequality} (as in~\cite[proof of Theorem 1.2]{aizenman2021depinning}), or an adaptation of the Ginibre correlation inequality~\cite{ginibre1970general} as detailed in~\cite{frohlich1978correlation, kharash2017fr} or~\cite[Proposition 2.3]{garban2023invisibility}.
\end{itemize}

\begin{figure} 
\begin{center}
\includegraphics[width=10cm]{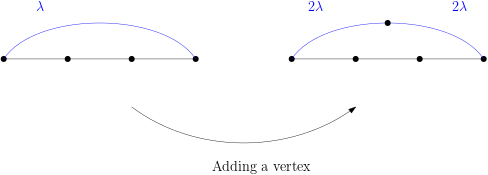}
\caption{Adding a vertex in the middle of an edge. This operation reduces the variance of linear functionals of the field.} \label{adition1.picture}
\end{center}
\end{figure}

Proposition~\ref{prop.prop2.9} can be iterated so as to obtain the following corollary which is one of the main ingredients of the proofs below (see Figure~\ref{AddingNvertices.picture}).

\begin{proposition} \label{corollary:splittingintoNparts}
 Consider an integer $N \in \N$ and an edge $yz \in E$, we denote by $\tilde G_N$ the graph obtained by adding $N$ new vertices $t_1, \ldots, t_N$ on the edge $yz$ (see Figure~\ref{AddingNvertices.picture}). 
    We denote by $\tilde \lambda$ the collection of conductances on the graph $\tilde G_N$ defined as follows:
\begin{itemize}
\item On the edges of $\tilde G_N$ which are not incident to the vertices $t_1 , \ldots, t_N$, we set $\tilde \lambda = \lambda$,
\item On the edges of $\tilde G_N$ which are incident to at least one of the vertices $t_1 , \ldots, t_N$, we set $\tilde \lambda = (N+1) \lambda$.
\end{itemize}
Then 
\begin{equation*}
         \forall ~ (i , j) \in V \times V,~\mathrm{Var}_{\tilde G , \tilde \lambda}^{\mathrm{IV-GFF}}\left[  \varphi(i) - \varphi(j)  \right] \leq \mathrm{Var}_{G , \lambda}^{\mathrm{IV-GFF}}\left[ \varphi(i) - \varphi(j) \right].
\end{equation*}
In particular, by setting $j = x$,
\begin{equation*} 
    \forall ~ i \in  V,~\mathrm{Var}_{\tilde G , \tilde \lambda}^{\mathrm{IV-GFF}}\left[  \varphi(i)  \right] \leq \mathrm{Var}_{G , \lambda}^{\mathrm{IV-GFF}}\left[ \varphi(i) \right].
\end{equation*}
\end{proposition}

\begin{figure} 
\begin{center}
\includegraphics[width=10cm]{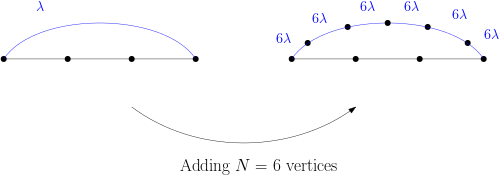}
\caption{Adding $N$ vertices in the middle of an edge (with here $N = 5$). This operation reduces the variance of linear functionals of the field} \label{AddingNvertices.picture}
\end{center}
\end{figure}

\subsection{Power-law interactions}

In this section, we collect two properties which play an important role in the analysis of the delocalisation of the $q$-SOS long-range chain: the decomposition of the power-law interaction as a mixture of Gaussian interactions and the (standard) FKG inequality for collections of independent random variables.

\subsubsection{Gaussian decomposition of power-law interactions} \label{sec:section2.5.1}

We state here the decomposition of the power-law interaction $x \mapsto e^{-|x|^q}$ (with $q \in (0 , 2]$) into a mixture of centred Gaussian densities. This decomposition is a well-known result for which we refer to the article of Penson and G\'{o}rska~\cite{penson2010exact} (and the references therein). 

We note that similar decompositions have played an important role in several contributions on random interfaces, for instance in the articles of Biskup-Koteck\'{y}~\cite{biskup2007phase}, Biskup-Spohn~\cite{biskupspohn} and Brydges-Spencer~\cite{brydges2012fluctuation}, and more recently in the article of Buchholz~\cite{buchholz2021phase}, Armstrong-Wu~\cite{armstrong2023scaling}, and Sellke~\cite{sellke2024localization} to study real-valued height functions, and of van Engelenburg-Lis~\cite{van2023duality} and Aizenman-Harel-Peled-Shapiro~\cite{aizenman2021depinning} to study integer-valued height functions.

\begin{proposition}[Gaussian decomposition of power-law interactions] \label{prop:Gaussdecomp}
For any exponent $q \in (0 , 2]$, there exists a Borel probability measure $\mu_q$ on $(0 , \infty)$ such that, for any $x \in \R$,
\begin{equation} \label{eq:defannealedgaussian}
    e^{- |x|^q} = \int_{(0 , \infty)} e^{- \lambda x^2} d \mu_q (\lambda).
\end{equation}
\end{proposition}

\begin{remark} \label{remark2.13}
    \begin{itemize}
    \item In the case of the integer-valued Gaussian free field $(q = 2)$, we have $\mu_2 = \delta_1$ (the Dirac measure at $1$).
    \item For $q \in (0,2)$, the measure $\mu_q$ is uniquely characterized by~\eqref{eq:defannealedgaussian}, it is absolutely continuous with respect to the Lebesgue measure, i.e., $d\mu_q (\lambda) = g_q(\lambda) d\lambda$, and the function $g_q$ has the following features (see~\cite{mikusinski1959function})
    \begin{equation*}
        g_q(\lambda) \sim_{\lambda \to 0} K_q x^{\frac{2 + q/2}{2 - q}} \exp \left( - A_q \lambda^{-\frac{q}{2-q}} \right) ~~\mbox{and} ~~ g_q(\lambda) \sim_{\lambda \to \infty} M_q \lambda^{-1 - \frac{q}{2}},
    \end{equation*}
    for some explicit constants $K_q, A_q, M_q \in (0 , \infty).$ These properties imply the following tail estimates on the measure $\mu_q$: there exists a constant $C := C(q) < \infty$ such that
    \begin{equation*}
        \forall K \in (1 , \infty), ~\mu_q\left(\left[0 , \frac{1}{K}\right] \right) \leq C \exp \left( - \frac{K^{\frac{q}{2-q}}}{C} \right) ~~\mbox{and}~~ \mu_q\left(\left[K , \infty \right) \right) \leq \frac{C}{K^{\frac{q}{2}}}. 
    \end{equation*}
    \item In the case of the SOS-model ($q=1$), the following formula holds
    \begin{equation*}
g_1(x) = \frac{e^{-1/4x}}{2x\sqrt{\pi x}}  .
  \end{equation*}
    \end{itemize}
\end{remark}

\subsubsection{The FKG inequality}
In this section, we state the standard FKG inequality for collections of independent real-valued random variables.

\begin{proposition}[Fortuin-Kasteleyn-Ginibre inequality] \label{prop.FKGinequality}
For $n \in \N$, let $\lambda_1 , \ldots, \lambda_n$ be independent nonnegative random variables, and let $f , g : \R^n \to [0 , \infty)$ be increasing (resp.\ decreasing) functions (see~\eqref{def.increasingfunction}). Then one has the inequality
\begin{equation} \label{eq:ineqFKG}
    \E \left[ f(\lambda_1 , \ldots, \lambda_n ) g(\lambda_1 , \ldots, \lambda_n ) \right] \geq  \E \left[ f(\lambda_1 , \ldots, \lambda_n ) \right]  \E \left[ g(\lambda_1 , \ldots, \lambda_n ) \right].
\end{equation}
\end{proposition}

\begin{remark}
    The inequality is stated for increasing functions depending on finitely many independent random variables but also holds for increasing functions depending on infinitely many independent random variables.
\end{remark}

\subsection{Convention for constants}
Throughout this article, the symbols $c$ and $C$ denote positive
constants which may vary monotonically from line to line, with $C$ increasing larger than $1$ and $c$ decreasing smaller than $1$. In Section~\ref{sec:section3}, they may only depend on the range exponent $\alpha$ (and will sometimes depend on the inverse temperature $\beta$ in which case this dependence will be made explicit and we will write $C_\beta$ and $c_\beta$ instead of $C$ and $c$). In Section~\ref{sec.section4qSOS}, we allow these constants to depend on the range exponent $\alpha$, the exponent $q$ and the inverse temperature~$\beta$.

\section{Delocalisation for the integer-valued Gaussian long-range chain} \label{sec:section3}

This section is devoted to the proof of Theorem~\ref{thm.gaussian} and is split into two subsections: Section~\ref{sec:lowerboundGaussian} is devoted to the lower bounds
and Section~\ref{sec:upperboundGaussian} is devoted to the upper bounds.
Each subsection is split into several cases, depending on the value of the exponent~$\alpha$.
We fix a large integer $N \in \N$ throughout the proof.

\subsection{The lower bounds} \label{sec:lowerboundGaussian}

\begin{figure} 
\begin{center}
\includegraphics[height =23cm, width=18cm]{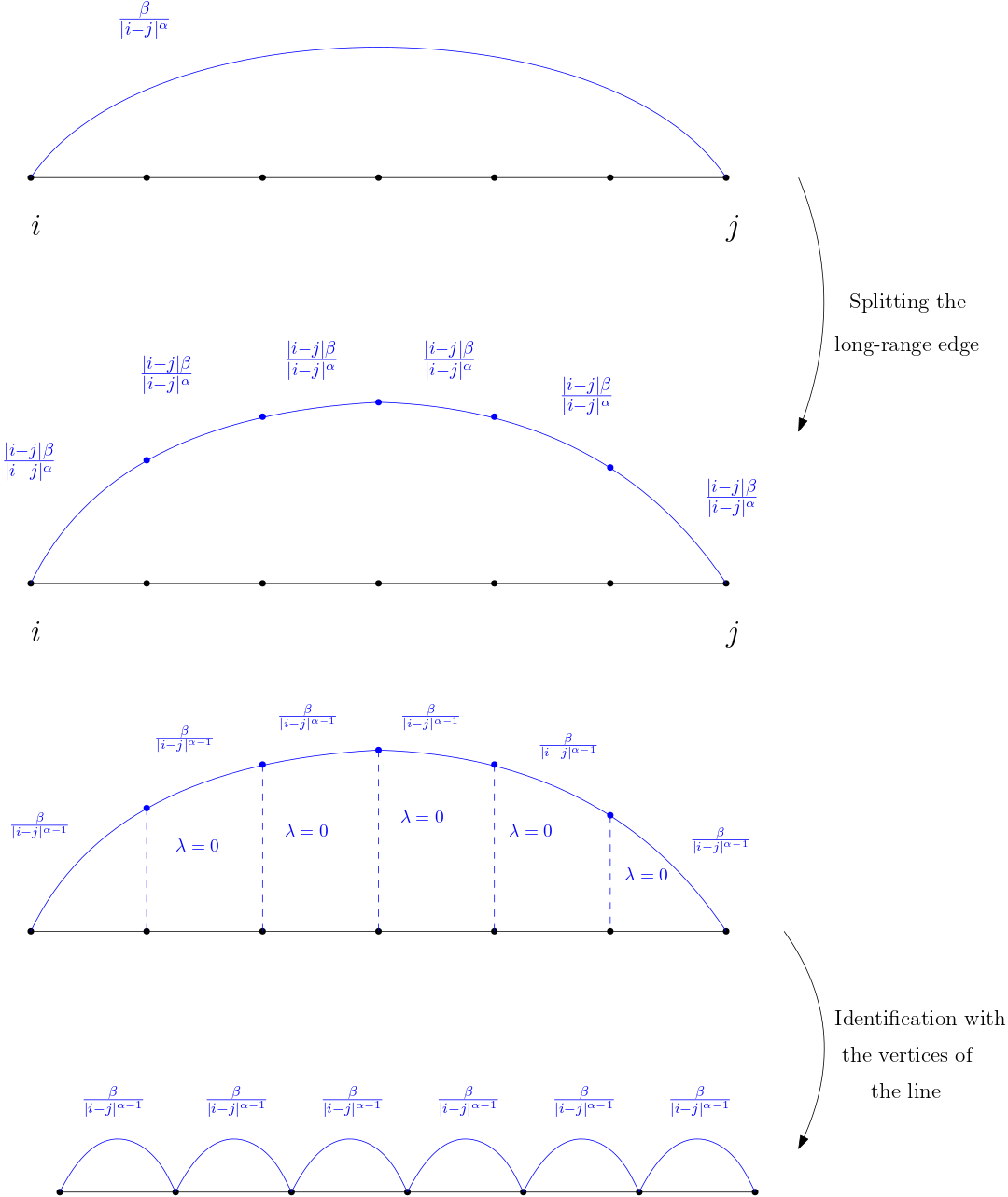}
\caption{Proof of the delocalisation when $\alpha > 3$} \label{Delocalphalarger3.picture}
\end{center}
\end{figure}

\subsubsection{The case $\alpha > 3$} \label{sec:lowerboundGaussianalpha>3}

\medskip

In this setting, we consider two integers $i , j \in \Z$ such that $|i - j| > 1$ and perform the following operations which all reduce the variance of the height $\varphi(0)$ (we refer to Figures~\ref{Delocalphalarger3.picture} and~\ref{Delocalphalarger32.picture} for guidance):
\begin{itemize}
    \item We first consider the edge connecting the vertices $i$ and $j$. This edge has a conductance equal to $\beta |i - j|^{-\alpha}$. Using Proposition~\ref{corollary:splittingintoNparts}, we add $(|i - j| - 1)$ vertices on this edge and multiply the conductances by $|i - j|$ (N.B. this operation reduces the variance of the height $\varphi(0)$ by Proposition~\ref{corollary:splittingintoNparts}). We denote by $i_1 , \ldots , i_{|i - j| - 1}$ the new vertices. The conductance of any edge incident to one of the vertices $i_1 , \ldots , i_{|i - j| - 1}$ is thus equal to $\beta |i - j|^{1-\alpha}$. We then apply Proposition~\ref{corollary.identification} to identify the vertex $i_k$ with $i + k$, for any $k \in \{ 1 , \ldots, |i - j| - 1 \}.$ These operations have the effect of erasing the edge $ij$ while increasing the values of the nearest neighbour conductances between $i$ and $j$ by the additive constant $\beta |i - j|^{1-\alpha}$.
    \item We then iterate this operation with all pairs of non-nearest neighbour vertices $i , j \in \Z$.
\end{itemize}
After performing all these operations, we obtain a nearest neighbour Gaussian chain on $\{ -N , \ldots, N\}$. We next consider a vertex $i \in \{ -N, \ldots, N - 1 \}$ and its right-neighbour $i+1$. After performing the operations mentioned above, the value of the conductance on the edge $i (i +1)$ is equal to 
\begin{equation*}
 \beta \sum_{j \leq i} \sum_{j' \geq i + 1} \frac{1}{\left| j - j'\right|^{\alpha - 1}}.
\end{equation*}
We next note that, for any fixed $j \leq i$ and since $\alpha - 1 > 2$, the following upper bound holds
\begin{equation*}
     \sum_{j' \geq i + 1} \frac{1}{\left| j - j'\right|^{\alpha - 1}} \leq \frac{C}{\left| i - j +1\right|^{\alpha -2}}.
\end{equation*}
Since $\alpha - 2 > 1$, we obtain
\begin{equation*}
    \sum_{j \leq i} \sum_{j' \geq i + 1} \frac{1}{\left| j - j'\right|^{\alpha - 1}} \leq C \sum_{j \leq i} \frac{C}{\left| i - j +1\right|^{\alpha -2}} = C \sum_{j = 1}^{\infty} \frac{C}{\left| j \right|^{\alpha -2}} \leq C.
\end{equation*}
The conductances of the nearest neighbour Gaussian chain are thus bounded uniformly in the parameter $N$. We finally perform a final operation on the chain (for which we refer to Figure~\ref{Delocalphalarger32.picture}): by Proposition~\ref{corollary.identification}, we identify the vertex $i$ with the vertex $-i$ for any $i \in \{ 1 , \ldots , N\}$. After performing this operation, we obtain a line of length $N$ on which all the conductances are bounded uniformly in $N$ (see Figures~\ref{Delocalphalarger32.picture} and~\ref{Delocalphalarger33.picture}).

\begin{figure} 
\begin{center}
\includegraphics[width=12cm]{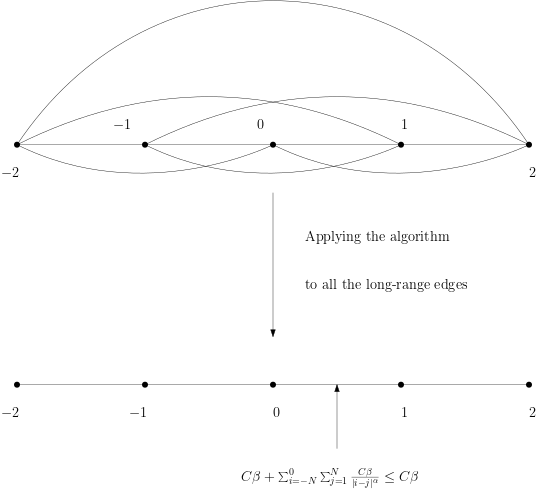}
\caption{Second part of the proof when $\alpha > 3$: we apply the algorithm to every long-range edge reduces the problem to a nearest neighbour Gaussian chain with finite conductances} \label{Delocalphalarger32.picture}
\end{center}
\end{figure}

At this stage, the variance of $\varphi(0)$ can be easily estimated using the identity~\eqref{eq:varianceforthechain} of Remark~\ref{rem:partitionfunctionmonotone} (with $\lambda = C \beta$). We thus have
\begin{equation*}
    \mathrm{Var}_{N , \beta , \alpha} \left[ \varphi(0) \right] \geq \mathrm{Var}_{N , C \beta} \left[ \varphi(0) \right] = c_\beta N.
\end{equation*}

\begin{figure} 
\begin{center}
\includegraphics[width=12cm]{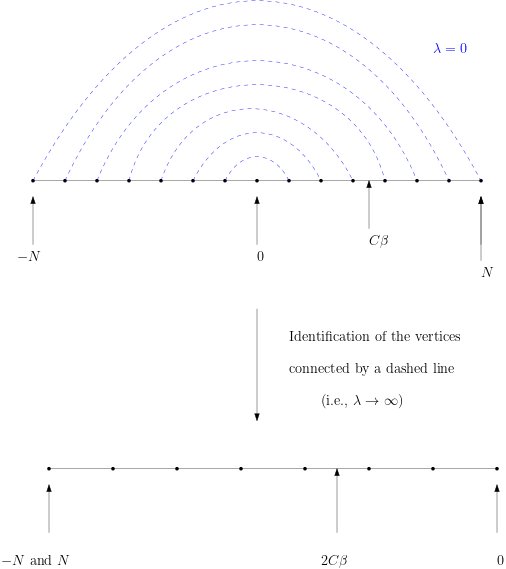}
\caption{Last step of the proof in the case $\alpha > 3$: we fold the nearest neighbour Gaussian chain to reduce the problem to a sum of independent random variables with finite variance.} \label{Delocalphalarger33.picture}
\end{center}
\end{figure}

\medskip

\subsubsection{The case $\alpha \in (2 , 3)$} \label{sec:section312}

\medskip

The argument in this setting is a refinement of the argument in the case $\alpha > 3$. We first set $\gamma := 3 - \alpha \in (0 , 1)$ and split the interval $\{ -N , \ldots, N \}$ into approximately $2N / N^\gamma = 2N^{\alpha -2}$ intervals of size $N^{\gamma}$. Specifically, we let $K := \lceil N^{\alpha -2} \rceil$, and, for any $k \in \{ -K , \ldots , K\}$, denote by
\begin{equation*}
    i_k := k \lceil N^\gamma \rceil.
\end{equation*}
For every pair of integers $i , j \in \Z$ such that $j > i$ and $|i - j| > 1$, we perform the following operations (we refer to Figures~\ref{Delocalphasmallerthan3.picture},~\ref{Delocalphasmallerthan31.picture} and~\ref{Delocalphasmallerthan32.picture} for guidance):
\begin{itemize}
    \item We consider the edge connecting the vertices $i$ and $j$. This edge has a conductance equal to $\beta |i - j|^{-\alpha}$. We then consider the number
    \begin{equation*}
            n_{ij} := \left| \left\{ i +1 , \ldots, j-1 \right\} \cap \{ i_{-K} , \ldots, i_K \} \right|.
    \end{equation*}
    In words, this is the number of vertices of the form $i_k$ which are strictly between $i$ and $j$ (N.B. this number can be equal to $0$). Note that we always have $n_{ij} \leq (2K+1)$ and that, if $n_{ij} \geq 2$, then 
    \begin{equation} \label{eq:estnij}
        c \frac{|i-j|}{N^\gamma} \leq n_{ij} \leq  \frac{ C |i-j|}{N^\gamma}
    \end{equation}
    (N.B. the lower bound is not true for $n_{ij} = 0$, and the upper bound is not true for $n_{ij} = 1$ as two vertices may be close to each other but can be placed on the line $\Z$ such that there is a vertex of the set $\{ i_0 , \ldots, i_K \}$ between them).
    If $n_{ij} = 0$, then we stop the procedure here. If $n_{ij} > 0$, then we use Proposition~\ref{corollary:splittingintoNparts} to add $n_{ij}$ vertices on this edge and multiply the conductances by $n_{ij}+1$. We denote by $j_1 , \ldots , j_{n_{ij}}$ the new vertices. The conductance of any edge incident to one of the vertices $j_1 , \ldots , j_{n_{ij}}$ is thus equal to $\beta (n_{ij}+1) |i - j|^{-\alpha}$. We then apply Proposition~\ref{corollary.identification} and, for any $k \in \{2 , \ldots, n_{ij} \}$, identify the vertex $j_k$ with the $k$-th vertex in the intersection $\left\{ i +1 , \ldots, j -1 \right\} \cap \{ i_{-K} , \ldots, i_K \}$.
    \item We then iterate this operation with all pairs of non-nearest neighbour vertices $i , j \in \{ -N , \ldots, N\}$.
\end{itemize}
The graph obtained after performing these operations take the form of the one depicted on Figure~\ref{Delocalphasmallerthan31.picture} and satisfies the following crucial property: all the long-range interactions lie inside the intervals $$\{ i_{-K} , \ldots, i_{-K +1} \},\ldots, \{ i_{K-1} , \ldots, i_K \}.$$

\begin{figure} 
\begin{center}
\includegraphics[height =16cm, width=14cm]{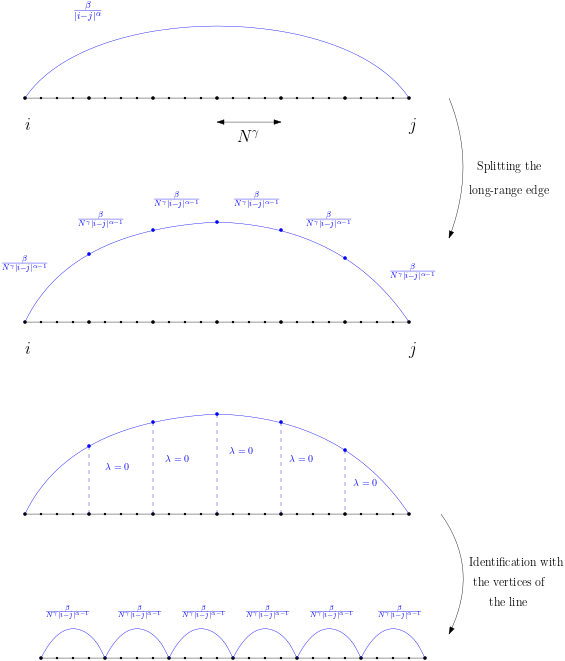}
\caption{First step of the proof in the case $\alpha \in (2 , 3)$. The long-range edges are split when the interval $\{ i , \ldots, j\}$ contains vertices of the form $k \lceil N^\gamma \rceil$ for $k \in \Z$.} \label{Delocalphasmallerthan3.picture}
\end{center}
\end{figure}

\begin{figure} 
\begin{center}
\includegraphics[width=11cm]{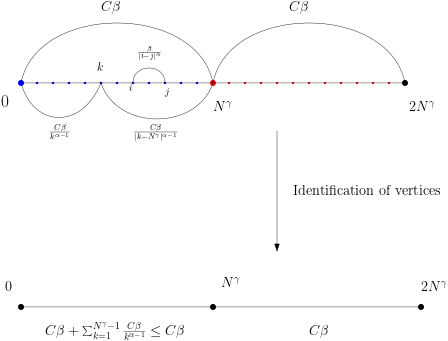}
\caption{Last step of the proof in the case $\alpha \in (2 , 3)$: we identify all the vertices in the same interval to reduce the problem to a nearest neighbour Gaussian chain. In the picture, all the blue vertices are identified together. Similarly, all the red vertices are identified together. These operations generate a graph containing two vertices drawn at the bottom} \label{Delocalphasmallerthan31.picture}
\end{center}
\end{figure}

\begin{figure} 
\begin{center}
\includegraphics[width=12cm]{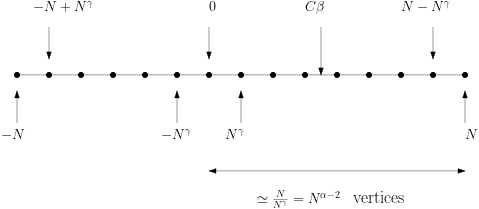}
\caption{Last step of the proof in the case $\alpha \in (2 , 3)$: the nearest neighbour Gaussian chain has finite conductance and its length is of order $N^{\alpha - 2}$.} \label{Delocalphasmallerthan32.picture}
\end{center}
\end{figure}
We next estimate the value of the conductances in each interval of the form $\{ i_{-k} , \ldots, i_{-k +1} \}$. To ease the notation, we will write the proof in the case of the interval $\{ i_{0} , \ldots, i_{1}  \} = \{ 0 , \ldots, \lceil N^\gamma \rceil \} $. We start with the (most important) conductance (see Figure~\ref{Delocalphasmallerthan31.picture}): the one associated with the edge connecting the two extremal sites $0$ and $\lceil N^\gamma \rceil$ of the interval $\{ 0 , \ldots, \lceil N^\gamma \rceil \}$. The value of this conductance is given by the formula
\begin{equation*}
     \beta \sum_{i \leq 0} \sum_{j \geq \lceil N^\gamma \rceil} \frac{(n_{ij}+1)}{|i - j|^\alpha}.
\end{equation*}
We then estimate the previous display by splitting the sum over the integers $j\in \{ \lceil N^\gamma \rceil, \ldots, N \}$ (for which the inequality~\eqref{eq:estnij} can be applied, noting that, since the distance between $i$ and $j$ is larger than $N^\gamma$, the upper bound holds even when $n_{ij} = 1$) and the integers $j \geq N$ (for which we use the bound $n_{ij} \leq (2K+1)$ with $K \simeq N^{\alpha -2}$). We obtain, for any $i \in \{ - N , \ldots, 0\}$, 
\begin{align} \label{eq:ineqnijestimate}
    \sum_{j \geq \lceil N^\gamma \rceil} \frac{(n_{ij}+1)}{|i - j|^\alpha} & \leq \frac{C}{N^\gamma} \sum_{j = \lceil N^\gamma \rceil}^N \frac{1}{|i - j|^{\alpha-1}} + C K \sum_{j \geq N} \frac{1}{|i - j|^{\alpha}} \\
    & \leq \frac{C}{N^\gamma}  \frac{1}{|i - N^\gamma|^{\alpha-2}} + \frac{C K}{N^{\alpha-1}} \notag \\
    & \leq \frac{C}{N^\gamma}  \frac{1}{|i - N^\gamma|^{\alpha-2}} + \frac{C}{N} \notag \\
    & \leq \frac{C}{N^\gamma}  \frac{1}{|i - N^\gamma|^{\alpha-2}}, \notag
\end{align}
where we used $N^{\gamma(\alpha -1)} \leq N$ in the last inequality.
Similarly, for any $i \leq -N$ (using this time only the bound $n_{ij} \leq (2K+1)$),
\begin{align*}
    \sum_{j \geq \lceil N^\gamma \rceil} \frac{(n_{ij}+1)}{|i - j|^\alpha} & \leq C K \sum_{j = \lceil N^\gamma \rceil}^{|i|} \frac{1}{|i|^\alpha} + C K \sum_{j \geq |i|} \frac{1}{|j|^\alpha} \\
    & \leq \frac{C K}{|i|^{\alpha-1}}.
\end{align*}
We finally sum over the negative integers and obtain
\begin{align*}
    \sum_{i \leq 0} \sum_{j \geq \lceil N^\gamma \rceil} \frac{(n_{ij}+1)}{|i - j|^\alpha} & \leq \frac{C}{N^\gamma} \sum_{i = - N}^0 \frac{1}{|i - N^\gamma|^{\alpha-2}}  + C K \sum_{i \leq - N} \frac{1}{|i|^{\alpha-1}}\\
    & \leq   \frac{C}{N^\gamma}  \sum_{i = - \lceil N^\gamma \rceil}^{0} \frac{1}{N^{\gamma(\alpha-2)}} +  \frac{C}{N^\gamma}  \sum_{i = -N }^{- \lceil N^\gamma \rceil} \frac{1}{|i|^{\alpha-2}} + C K \sum_{i \leq - N} \frac{1}{|i|^{\alpha-1}} \\
    & \leq \frac{C}{N^{(\alpha - 2) \gamma}}  +  C N^{3-\alpha-\gamma} + \frac{CK }{ N^{\alpha-2}} \\
    & \leq C,
\end{align*}
where in the second inequality we used that $|j - N^\gamma| \geq N^\gamma$ for any $j \leq 0$ and $|j - N^\gamma| \geq c|j|$ for any $j \leq - \lceil N^\gamma \rceil$, and in the last inequality, we used the definition $\gamma := 3-\alpha$.
This implies that the conductance of the edge connecting the extremal sites $0$ and $\lceil N^\gamma \rceil$ is bounded uniformly in~$N$.

We then fix an integer $k \in \{ 1 , \ldots, \lceil N^\gamma \rceil -1 \}$ and estimate the value of the conductance between $k$ and $\lceil N^\gamma \rceil$ (see Figure~\ref{Delocalphasmallerthan31.picture}). The value of this conductance is given by the identity
\begin{equation*}
    \beta \sum_{j \geq \lceil N^\gamma \rceil} \frac{(n_{kj}+1)}{|j - k|^\alpha}.
\end{equation*}
We estimate this term by splitting the sum over the integers $j \in \{ \lceil N^\gamma \rceil , \ldots , 2 \lceil N^\gamma \rceil\}$ (for which $n_{kj} = 1$), and the integers $j \in \{ 2 \lceil N^\gamma \rceil , \ldots , N \}$ (for which $n_{kj} \geq 2$ and~\eqref{eq:estnij} can be applied) and the integers $j \geq N$ (for which the inequality $n_{kj} \leq (2K+1)$ can be applied). We obtain
\begin{align} \label{eq:boundlinbkbound}
    \sum_{j \geq \lceil N^\gamma \rceil} \frac{(n_{kj}+1)}{|j - k|^\alpha} & \leq \sum_{j = \lceil N^\gamma \rceil}^{2 \lceil N^\gamma \rceil} \frac{C}{|j - k|^\alpha} + \frac{C}{N^\gamma} \sum_{j = 2 \lceil N^\gamma \rceil}^{N} \frac{1}{|j - k|^{\alpha-1}} + CK \sum_{j = N}^\infty \frac{1}{j^\alpha}  \\
    & \leq \frac{C}{\left| \lceil N^\gamma \rceil - k \right|^{\alpha -1}} +  \frac{C}{N^\gamma} \sum_{j = 2 \lceil N^\gamma \rceil}^{N} \frac{1}{|j|^{\alpha-1}} + CK \frac{1}{N^{\alpha-1}} \notag \\
    & \leq \frac{C}{\left| \lceil N^\gamma \rceil - k \right|^{\alpha -1}} + \frac{1}{N^\gamma} \frac{C}{N^{\gamma (\alpha-2)}} + \frac{C}{N} \notag \\
    & \leq \frac{C}{\left| \lceil N^\gamma \rceil - k \right|^{\alpha -1}}, \notag
\end{align}
where in the second line, we used that $|j - k| \geq c |j|$ when $j \geq 2 \lceil N^\gamma \rceil$ and in the last line, we used that $\left| \lceil N^\gamma \rceil - k \right| \leq N^\gamma$ and $N^{\gamma(\alpha -1)} \leq N$.

We finally use Proposition~\ref{corollary.identification} to identify all the vertices $\{ 0 , \ldots, \lceil N^\gamma \rceil - 1 \}$ together. This operation transforms the interval $\{ 0 , \ldots, \lceil N^\gamma \rceil \}$ into a graph containing two vertices connected by a conductance whose value is bounded by
\begin{equation*}
   C \beta + \sum_{k = 1}^{\lceil N^\gamma \rceil - 1} \frac{C \beta}{\left| \lceil N^\gamma \rceil - k \right|^{\alpha -1}} \leq C \beta
\end{equation*}
where we used that $\alpha -1 > 1$ (see Figures~\ref{Delocalphasmallerthan31.picture} and~\ref{Delocalphasmallerthan32.picture}). This conductance is thus bounded uniformly in $N$. Applying the same operation to all the intervals the form $\{ i_{-k} , \ldots, i_{-k +1} \}$ yields a nearest neighbour Gaussian chain of length $K \simeq N^{\alpha - 2}$ whose conductances are all bounded by $C \beta$. We may thus conclude as in the previous case that
\begin{equation*}
    \mathrm{Var}_{N , \beta , \alpha} \left[ \varphi(0) \right] \geq \mathrm{Var}_{K , C \beta} \left[ \varphi(0) \right] \geq c_\beta N^{\alpha - 2}.
\end{equation*}

\medskip

\subsubsection{The case $\alpha = 3$}

\medskip

This case is in fact similar to the case $\alpha \in (2 , 3)$, the main difference is that we partition the interval $\{ -N , \ldots, N \}$ into approximately $2 N / \ln N$ intervals of size $ \ln N$ (instead of $2 N/ N^\gamma$ intervals of size $N^\gamma$). We thus only point out the main differences with the proof of Section~\ref{sec:section312}. For the conductance of the edge connecting the two extremities $0$ and $\lceil \ln N \rceil$ of the interval $\{ 0 , \ldots, \lceil \ln N \rceil \}$, we have the identity
\begin{equation*}
   \beta \sum_{i \leq 0} \sum_{j \geq \lceil \ln N \rceil} \frac{(n_{ij}+1)}{|i - j|^3}.
\end{equation*}
This term can then be estimated as follows: for any $i \in \{ - N , \ldots, 0\}$,
\begin{equation*}
    \sum_{j \geq \lceil \ln N \rceil} \frac{(n_{ij}+1)}{|i - j|^3} \leq \frac{C}{\ln N} \frac{1}{|i - \ln N|},
\end{equation*}
and for any $i \leq -N$,
\begin{equation*}
    \sum_{j \geq \lceil \ln N \rceil} \frac{(n_{ij}+1)}{|i - j|^3} \leq \frac{C N}{\ln N} \frac{1}{|i|^2}.
\end{equation*}
Summing over the negative integers, we deduce that
\begin{align*}
    \sum_{i \leq 0} \sum_{j \geq \lceil \ln N \rceil} \frac{(n_{ij}+1)}{|i - j|^3} & \leq \frac{C }{\ln N} \sum_{i = -N}^{ 0}\frac{1}{|i - \ln N|} + \sum_{i \leq -N} \frac{C N}{\ln N|i|^{2}} \\
    & \leq \frac{C}{\ln N} \sum_{i = - \lceil \ln N \rceil}^{0} \frac{1}{|\ln N|} + \frac{C}{\ln N} \sum_{i = -N}^{- \lceil \ln N \rceil } \frac{1}{|i|} + \frac{C N}{\ln N}  \sum_{i \leq -N} \frac{1}{|i|^{2}}\\
    & \leq C.
\end{align*}
The rest of the proof is then similar to the case $\alpha \in (2 , 3)$. Using the same computation as in~\eqref{eq:boundlinbkbound}, we obtain that, for any $k \in \{ 1 , \ldots, \lceil \ln N \rceil -1 \}$, the conductance between the vertices $k$ and $\lceil \ln N \rceil$ is upper bounded by the value $C \beta / ( \lceil \ln N \rceil - k)^2$.  We finally use Proposition~\ref{corollary.identification} to identify all the vertices $\{ 0 , \ldots, \lceil \ln N \rceil - 1 \}$ together. This operation transforms the graph $\{ 0 , \ldots, \lceil \ln N \rceil - 1 \}$ into a graph containing two vertices connected by a conductance whose value is bounded by $C \beta$.

Applying the same series of operations to all the intervals of the form $\{ i_{-k} , \ldots, i_{-k +1} \}$, we obtain a nearest neighbour Gaussian chain of length $N / \ln N$ whose conductances are all bounded by $C \beta$. We may thus conclude as in the previous cases.

\subsubsection{The case $\alpha = 2$ at high temperature}

In this section, we show that when the range exponent $\alpha$ is equal to $2$ and the inverse temperature $\beta$ is sufficiently small, the variance of the height $\varphi(0)$ grows at least like the logarithm of the length of the chain. We note that two proofs of this result have already been obtained: one by Kjaer and Hilhorst~\cite{KH82} (using a duality transformation for the Gaussian chain) and one by Garban~\cite{garban2023invisibility} (where the scaling limit of the chain is identified; in fact the proof written below shares strong similarities with his argument).

We will in fact prove an inequality which is valid for any inverse temperature $\beta$ and relates the variance of the integer-valued Gaussian chain with $\alpha = 2$ at inverse temperature $\beta$ to the one of a two-dimensional integer-valued Gaussian free field at inverse temperature $C_0\beta$ (for some explicit constant $C_0 > 1$, see Proposition~\ref{prop:prop4.1} below). This second model is known to undergo a localisation/delocalisation phase transition~\cite{frohlich1981kosterlitz, kharash2017fr, lammers2022height, van2023elementary, aizenman2021depinning}. Combined with Proposition~\ref{prop:prop4.1}, these results imply the delocalisation of the Gaussian chain at high temperature.

We first introduce the version of the two-dimensional integer-valued Gaussian free field used in the statement and proof of Proposition~\ref{prop:prop4.1}. In order to state the definition, we fix an integer $N \in \N$ and introduce the notation:
\begin{itemize}
\item We let $\Lambda_N := \left\{ - N , \ldots , N \right\} \times \{ -2N , \ldots, 2N \} \subseteq \Z^2$ be a two-dimensional rectangle of short and long side lengths $N$ and $2N$. 
\item We let $\partial \Lambda_N := \left\{ x \in \Z^2 \setminus \Lambda_N \, : \, \exists y \in \Lambda_L, x \sim y \right\}$ be the external vertex boundary of $\Lambda_L$ (writing $x \sim y$ to mean that $x , y \in \Z^2$ are nearest neighbours), and set $\Lambda_N^+ := \Lambda_N \cup \partial \Lambda_N$.
\item We let $\Omega_{\Lambda_N} := \left\{ \varphi : \Lambda_N^+ \to \Z \, : \, \varphi = 0 ~\mbox{on}~ \partial \Lambda_N  \right\}$ be the set of integer-valued height functions on the box $\Lambda_N$ with Dirichlet boundary condition.
\end{itemize}

\begin{definition}[Two-dimensional integer-valued Gaussian free field] \label{def:2DGFF}
Given an integer $N \in \N$ and an inverse temperature $\beta$, we define the two dimensional integer-valued Gaussian free field at inverse temperature~$\beta \in (0 , \infty)$ to be the probability distribution on $\Omega_{\Lambda_N}$ given by Definition \ref{def:IV-GFF} where the graph is $\Lambda_N$, the vertices of $\partial\Lambda_N$ are identified and assigned the role of root, and the conductances all equal $\beta$. We denote by $\mathrm{Var}_{N , \beta}^{\mathrm{2D-IVGFF}}$ the variance with respect to this measure.
\end{definition}

We are now able to state the main result of this section.

\begin{proposition} \label{prop:prop4.1}
	There exists a constant $C_0 \in (1 , \infty)$ such that, for any inverse temperature $\beta \in (0, \infty)$,
	\begin{equation} \label{eq:reductionofdimension}
	\mathrm{Var}_{N , C_0 \beta}^{\mathrm{2D-IVGFF}} \left[ \varphi(0) \right] \leq \mathrm{Var}_{N , \beta , 2} \left[ \varphi(0) \right].
	\end{equation}
\end{proposition}

\begin{remark}
    As mentioned above, the result of Fr\"{o}hlich and Spencer~\cite{frohlich1981kosterlitz} (or to be precise, the result of Wirth~\cite[Proposition 8]{wirth2019maximum} who extends the techniques of~\cite{frohlich1981kosterlitz} to include the Dirichlet boundary condition considered here) implies that there exists an inverse temperature $\beta_c > 0$ and a constant $c > 0$ such that, for any $\beta \leq \beta_c$,
    \begin{equation*}
        \frac{c}{\beta} \ln N \leq \mathrm{Var}_{N , C_0 \beta}^{\mathrm{2D-IVGFF}} \left[ \varphi(0) \right].
    \end{equation*}
    Combining this inequality with~\eqref{eq:reductionofdimension}, we obtain that, for $\beta$ sufficiently small,
    \begin{equation*}
        \frac{c}{\beta} \ln N \leq \mathrm{Var}_{N , \beta , 2} \left[ \varphi(0) \right].
    \end{equation*}
\end{remark}

\begin{proof}
	We first introduce the following collection of non-negative numbers (N.B. the choice for these constants is not unique; they are chosen so that the inequalities~\eqref{eq:eq3.6} and~\eqref{eq:13410710} hold but any other choice satisfying these inequalities is admissible):
	\begin{equation*}
	\forall (k , l) \in \N \times \N ~\mbox{with}~ k \leq l, ~~ c_{k , l} := C_1 \sqrt{\frac{k}{l}} \frac{\beta}{l}>0,
	\end{equation*}
	where $C_1 \in (1 , \infty)$ is a constant chosen so that the following two inequalities are satisfied (the choice of the value of $C_1$ affects the second inequality):
	\begin{itemize}
		\item For any $k \in \N$,
		\begin{equation} \label{eq:eq3.6}
		\sum_{l = k}^\infty c_{k,l} = C_1 \beta \sqrt{k} \sum_{l = k}^\infty \frac{1}{l \sqrt{l}} \leq C_1 C_2 \beta,
		\end{equation}
        with $C_2 := \sup_{k \in \N} \sqrt{k} \sum_{l = k}^\infty \frac{1}{l \sqrt{l}} < \infty.$
		\item For any integer $l \in \N$,
		\begin{equation} \label{eq:13410710}
		\frac{l^2}{C_1} + 2 \sum_{k = 1}^l \frac{1}{c_{k,l}} = \frac{l^2}{C_1} + \frac{2}{C_1} l \sqrt{l} \sum_{k = 1}^l \frac{1}{\sqrt{k}} \leq \frac{l^2}{\beta}. 
		\end{equation}
	\end{itemize}
	We then perform a series of operations which reduce the variance of the height $\varphi(0)$, and map the Gaussian long-range chain (with $\alpha = 2$) to a two-dimensional integer-valued Gaussian free field.
	
	We first consider the two dimensional rectangle $\Lambda_N$ and identify the set $\left\{ -N , \ldots , N\right\}$ with the horizontal line $\left\{ -N , \ldots , N\right\} \times \{0 \} \subseteq \Lambda_N$. Then, for any pair of vertices $i , j \in \{ -N , \ldots, N \}$ with $i < j$, we add $3  (j - i) - 1 $ vertices on the edge connecting $i$ to $j$, denote these vertices by $x_1 , \ldots, x_{3 (j - i) - 1}$ and set $x_0 = i$ and $x_{3(j - i)} = j$ (see Figure~\ref{Splittialphacritical.picture}).
	
	We then assign the following conductances to the edges (of the form $(x_k, x_{k+1})$) created by the previous step (see Figure~\ref{alpha=2critical2.picture}):
	\begin{itemize}
		\item For $k \in \{ 1 , \ldots , (j - i) - 1\}$, we assign the conductance $c_{k ,  (j - i)}$ to the edge connecting $x_k$ and $x_{k+1}$;
		\item For $k \in \{ (j - i) , \ldots , 2 (j - i) - 1\}$, we assign the conductance $\frac{C_1 \beta}{(j - i)}$ to the edge connecting $x_k$ and $x_{k+1}$;
		\item For $k \in \{ 2 (j - i), \ldots , 3 (j - i) - 1\}$, we assign the conductance $c_{3 (j - i) - k , j-i}$ to the edge connecting $x_k$ and $x_{k+1}$.
	\end{itemize}
	Proposition~\ref{prop.prop2.9} and the inequality~\eqref{eq:13410710} ensure that performing these operations reduces the variance of the height $\varphi(0)$.
	We then identify the vertices $x_1 , \ldots, x_{3 (j - i) - 1}$ with the ones of the two dimensional box $\Lambda_N$ as follows (see Figure~\ref{alpha=2critical2.picture}):
	\begin{itemize}
		\item For $k \in \{ 1 , \ldots , (j - i)\}$, we identify the vertex $x_k$ with the vertex $(i , k) \in \Lambda_N$.
		\item For $k \in \{(j - i) , \ldots , 2 (j - i)\}$, we identify the vertex $x_k$ with the vertex $(i+k - (j - i) ,  (j - i) ) \in \Lambda_N$.
		\item For $k \in \{ 2(j - i) , \ldots , 3 (j - i)-1\}$, we identify the vertex $x_k$ with the vertex $(j , 3 (j - i) - k ) \in \Lambda_N$.
	\end{itemize}
	We then perform these operations on all pairs of vertices $i , j \in \{ -N , \ldots, N \}$ (N.B. when two vertices are identified to the same vertex of $\Lambda_N$, then they are identified together; see Figures~\ref{alpha=2critical3.picture},~\ref{alpha=2critical4.picture} and~\ref{alpha=2critical5.picture}). We then estimate the value of the conductances on each edge:
	\begin{itemize}
		\item On the horizontal edges (see Figure~\ref{alpha=2critical3.picture}): for $(k , l) \in \Lambda_N$, the conductance of the horizontal edge $\{ (k , l) , (k +1, l) \}$ receives a contribution of $\beta/l$ for each long-range edge of the form $(i , i+l)$ such that $i \leq k < i+l$. There are (at most) $l$ such contributions and we thus obtain
		\begin{equation*}
		c_{(k , l) , (k+1 , l)} \leq l \times \frac{C_1 \beta}{l} \leq C_1 \beta.
		\end{equation*}
		\item On the vertical edges: for $(k , l) \in \Lambda_N$, the conductance of the vertical edge  $\{ (k , l) , (k, l + 1) \}$, denoted by $c_{(k , l) , (k , l + 1)}$ below, receives a contribution from each long-range edge of the form $(k , n) \in \{-N , \ldots, N \} \times \{-N , \ldots, N \}$ with $|k-n| > l$, and each time the value of this contribution is $c_{l , |k - n|}$. We thus obtain
		\begin{equation*}
		c_{(k , l) , (k , l + 1)} \leq \sum_{\substack{n \in \{ - L , \ldots, L \} \\ |n - k| > l}} c_{l , |k - n|} \leq  2 \sum_{n = l}^\infty c_{l , n} \leq 2 C_1 C_2 \beta.
		\end{equation*}
	\end{itemize}
	
	We finally increase the value of all the conductances in the rectangle $\Lambda_{N}$ to $2 C_1 C_2 \beta$. This final operation reduces the value of the variance of $\varphi (0)$ (where $0$ now denotes the centre of the box $\Lambda_N$), and the model obtained after performing this final operation is the two dimensional integer-valued Gaussian free field in the rectangle $\Lambda_N$ with conductance equal to $2 C_1 C_2 \beta$. The proof of Proposition~\ref{prop:prop4.1} is thus complete (by setting $C_0 := 2 C_1 C_2$).
\end{proof}

\begin{figure} 
	\begin{center}
		\includegraphics[scale=0.6]{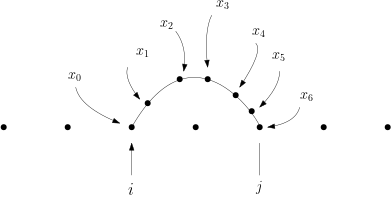}
		\caption{First step of the proof of Proposition~\ref{prop:prop4.1}. For $i , j \in \{-N , \ldots, N \}$ with $i < j$, $3 (j - i) - 1$ vertices are added to the edge connecting the vertices $i$ and $j$. On the figure, the value $j - i = 2$ was chosen and $5$ vertices are added to the edge connecting $i$ and $j$.} \label{Splittialphacritical.picture}
	\end{center}
\end{figure}

\begin{figure} 
	\begin{center}
		\includegraphics[scale=0.6]{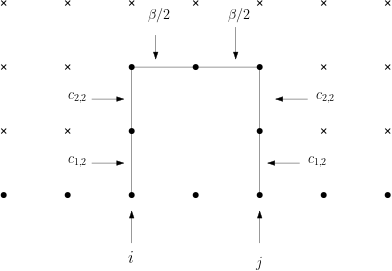}
		\caption{Second step of the proof of Proposition~\ref{prop:prop4.1}: conductances are assigned to the edges created at the first step and the vertices $x_1 , \ldots, x_{3 (j - i) - 1}$ are identified with vertices of the two-dimensional box $\Lambda_N$.} \label{alpha=2critical2.picture}
	\end{center}
\end{figure}

\begin{figure} 
	\begin{center}
		\includegraphics[scale=0.6]{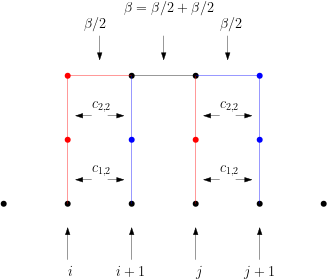}
		\caption{Contributions from the two long-range edges (drawn in red and blue) to the horizontal edge connecting the vertices $(0 , 2)$ and $(1,2)$ (drawn in black).} \label{alpha=2critical3.picture}
	\end{center}
\end{figure}

\begin{figure} 
	\begin{center}
		\includegraphics[scale=0.5]{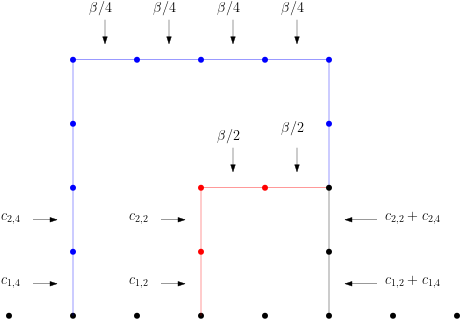}
		\caption{Contributions from two long-range edges (drawn in red and blue) to the two horizontal edges drawn in black.} \label{alpha=2critical4.picture}
	\end{center}
\end{figure}

\begin{figure} 
	\begin{center}
		\includegraphics[scale=0.6]{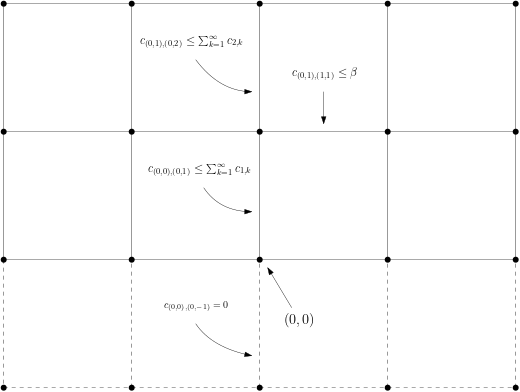}
		\caption{End of the proof of Proposition~\ref{prop:prop4.1}. The graph obtained is a two-dimensional box and the conductance of each edge is smaller than $C_0 \beta$ (and sometimes equal to $0$, as, for instance, for the edges in the bottom half of the box)} \label{alpha=2critical5.picture}
	\end{center}
\end{figure}

\subsubsection{The case $\alpha \in (1,2)$}

This case is the simplest: using the monotonicity of the variance in the length of the chain (as mentioned in Remark~\ref{rem:remarkmonotony}), we have
 \begin{equation*}
    c_\beta \leq \mathrm{Var}_{1 , \beta , \alpha} \left[ \varphi(0) \right] \leq \mathrm{Var}_{N , \beta , \alpha} \left[ \varphi(0) \right].
 \end{equation*}

\subsection{The upper bounds} \label{sec:upperboundGaussian}

This section is devoted to the proof of the upper bounds  of Theorem~\ref{thm.gaussian}. As in Section~\ref{sec:lowerboundGaussian}, we fix a large integer $N \in \N$ and split the proof into several cases, depending on the value of the exponent~$\alpha$.

\subsubsection{The case $\alpha > 3$} \label{sec:casealpha>3}

\medskip

This case is the simplest. Using Proposition~\ref{corollary.deletion.edges}, we may delete all the edges which are either non-nearest neighbour or whose endpoints are negative integers. The resulting graph is a line on which all the conductances take the value $\beta$. Applying the identity~\eqref{eq:varianceforthechain} of Remark~\ref{rem:partitionfunctionmonotone}, we obtain
\begin{equation*}
    \mathrm{var}_{N , \beta , \alpha} \left[ \varphi(0) \right]  \leq C_\beta N.
\end{equation*}

\medskip

\subsubsection{The case $\alpha = 3$ } \label{sec:casealpha=3}

\medskip

This case is more technical and a multiscale argument has to be implemented in order to identify the logarithmic correction. 

\begin{figure} 
\begin{center}
\includegraphics[height =17cm, width=14cm]{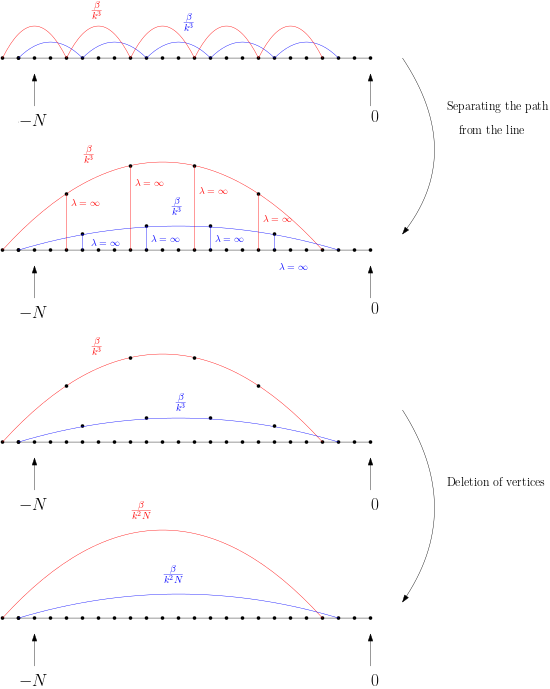}
\caption{First part of the proof of the upper bound when $\alpha = 3$. With the specific value $k =4$, the $2$ paths of the set $\Gamma_4$ are drawn (one in red and one in blue). We then isolate these two paths from the line $\mathbb{Z}$ using the monotonicity under identification of vertices. In the last step depicted, we use the monotonicity under addition/deletion of vertices to generate a long-range edge connecting the starting and end points of the paths considered.} \label{upperboundalpha=3firstpart.picture}
\end{center}
\end{figure}

We set $K := \lfloor \sqrt{N} \rfloor$. For each even integer $k \in \{ 2 , \ldots, \lfloor \sqrt{N} \rfloor \}$, we construct a collection of $(k/2 +1)$ paths starting from the interval $\{ - k , \ldots, - k/2 \}$ and making jumps of size $k$ to the left until they reach an integer smaller than $-N$. More precisely, for any integer $i \in \{ - k , \ldots, - k/2 \}$, we let $\gamma_{k , i}$ be the path starting from $i$ and making jumps of size $k$ until it reaches an integer smaller than $-N$.
We then denote by $\Gamma_k$ the set of paths constructed this way, i.e.,
\begin{equation*}
    \Gamma_k := \left\{ \gamma_{k , i} \, : \, i \in \{ -k , \ldots, -k/2 \} \right\}.
\end{equation*}
We refer to Figure~\ref{upperboundalpha=3firstpart.picture} for a visual description of set $\Gamma_{k}$ in the case $k = 4$. 
An important observation has to be made here: two paths of the form $\gamma_{k , i} $ and $\gamma_{k , i'} $ with $i \neq i'$ use distinct long-range edges of length $k$.
We then denote by $\Gamma$ the union of all the sets $\Gamma_k$ for $k \in \{ 2 , \ldots, \lfloor \sqrt{N} \rfloor \}$, i.e.,
\begin{equation*}
        \Gamma := \bigcup_{k = 2}^{\lfloor \sqrt{N} \rfloor} \Gamma_k.
\end{equation*}
Let us note that two paths of the form $\gamma_{k , i} $ and $\gamma_{k' , i'} $ must also use distinct long-range edges: indeed, if $k \neq k'$ then the path $\gamma_{k , i} $ uses only edges of length $k$ and the path $\gamma_{k' , i'} $ uses only edges of length $k'$, they are thus disjoint, and if $k = k'$, then we are in the case mentioned above.

We next perform the following operations on the path which all increase the variance of the height $\varphi(0)$ (and refer to Figure~\ref{upperboundalpha=3firstpart.picture} for guidance):
\begin{itemize}
    \item We first consider an integer $k \in \{ 2 , \ldots , \lfloor \sqrt{N} \rfloor  \}$ and an integer $i \in \{ - k  , \ldots, -k/2 \}$, and consider the path~$\gamma_{k , i}$. Applying Proposition~\ref{corollary.identification}, we may separate the vertices of $\gamma_{k , i}$ from the line $\Z$ as described in the first two steps of Figure~\ref{upperboundalpha=3firstpart.picture}. We next apply Proposition~\ref{corollary:splittingintoNparts} to erase all the isolated vertices on the long-range edges (see the last step of Figure~\ref{upperboundalpha=3firstpart.picture}). This operation generates an edge connecting the vertex $i$ to the endpoint of the path $\gamma_{k , i}$ with a conductance equal to
    \begin{equation*}
        c_{k , i} := \frac{1}{\sum_{j = 1}^{\lceil N/k \rceil} \frac{k^3}{\beta}}.
    \end{equation*}
    This conductance can be lower bounded as follows, for some constant $c \in (0,1)$,
    \begin{equation} \label{eq:lowerboundconductancecijk}
        c_{k , i} \geq \frac{c \beta}{N k^2}.
    \end{equation}
    \item We then iterate this operation to all the paths of $\Gamma$, this gives rise to $\left| \Gamma \right|$ new edges connecting an integer smaller than $-N$ (the endpoints of the paths of $\Gamma$) to a vertex close to $0$ (the starting points of the paths of $\Gamma$) and whose conductance are lower bounded by~\eqref{eq:lowerboundconductancecijk}.
\end{itemize}
We next tackle a technical difficulty: the starting point of the paths of $\Gamma$ is not the vertex $0$ (specifically, the starting point of the path $\gamma_{k , i} \in \Gamma$ is the vertex $i$). This is achieved thanks to the following argument (for which we refer to Figure~\ref{upperboundalpha=3secondpart.picture}).

\begin{figure} 
\begin{center}
\includegraphics[height =20cm, width=15cm]{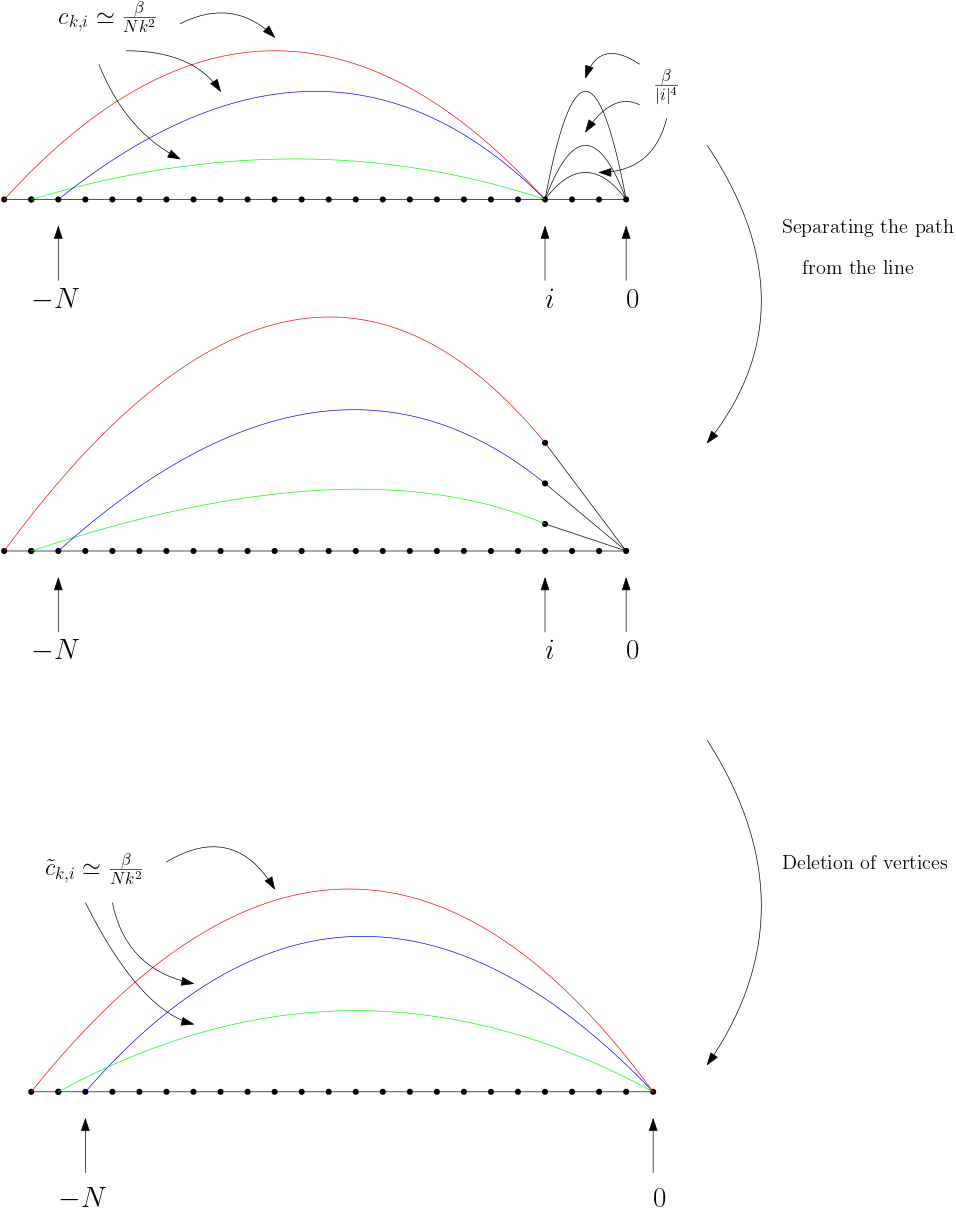}
\caption{Second part of the proof of the upper bound when $\alpha = 3$. We deal with the technical difficulty that the paths constructed in the previous step do not start from the vertex $0$ and, in the pictures, $3$ paths starting from $i$ are depicted. To handle this difficulty, we make use of the edge connecting $i$ to $0$, split it into the number of paths ending at $i$ and then isolate each of the paths.} \label{upperboundalpha=3secondpart.picture}
\end{center}
\end{figure}

\begin{figure} 
\begin{center}
\includegraphics[height =15cm, width=15cm]{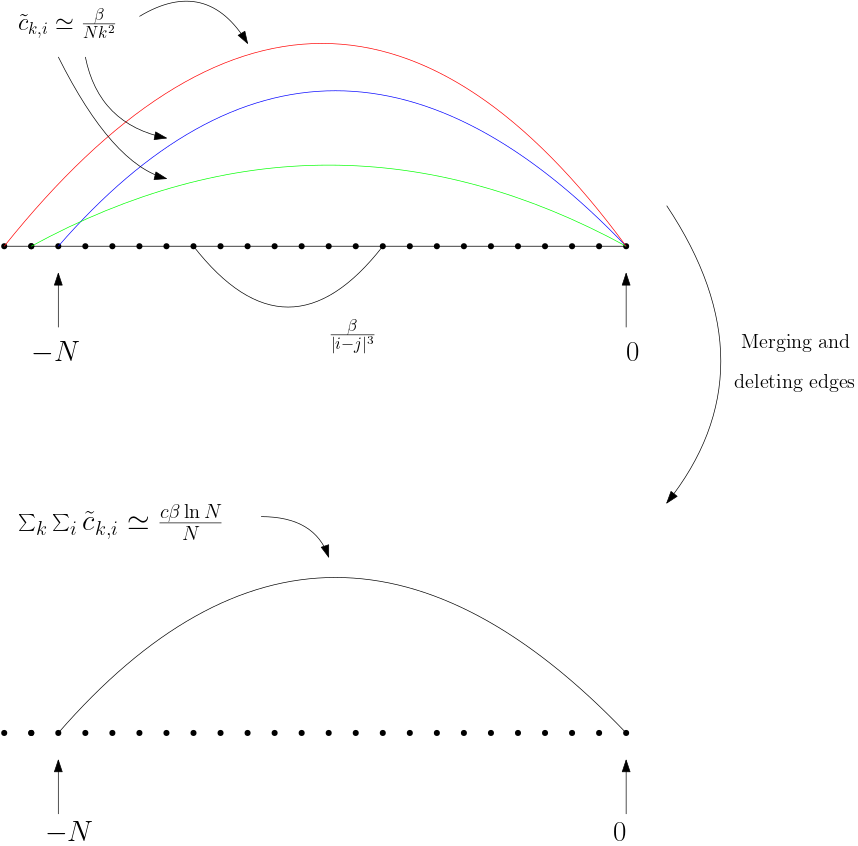}
\caption{Last part of the proof of the upper bound when $\alpha = 3$. We merge all the edges connecting the region $\{ i \in \Z \, : \, i \leq -N \}$ to $0$ which have been constructed in the previous step. This gives rise to a unique long-range edge connecting $-N$ to $0$ with a conductance of order $\beta \ln N / N.$} \label{upperboundalpha=3lastpart.picture}
\end{center}
\end{figure}

We first note that for any $i \in \{ - 1 , \ldots , - K \}$, the vertex $i$ can be the starting point of at most $|i|$ paths of the set $\Gamma$ (the paths $\left\{ \gamma_{k, i} \, : \, k \in \{ 2i , \ldots, i \} \cap 2\Z \right\}$). We then split the edge connecting $i$ to $0$ into $|i|$ edges with conductance equal to $\beta / |i|^4$ (see Figure~\ref{upperboundalpha=3secondpart.picture}), use Proposition~\ref{corollary.identification} to separate the paths $\left\{ \gamma_{k, i} \, : \, k \in \{ i , \ldots, 2i \} \cap 2\Z \right\}$ for the line $\Z$ and apply Proposition~\ref{corollary:splittingintoNparts} to erase the isolated vertices. These operations all increase the variance of the height $\varphi(0)$ and give rise to a  collection of long-range edges connecting the vertex $0$ to the region $\{ x \in \Z \, : \, x \leq -N \}$  whose conductances, denoted by $\tilde c_{k , i}$, are equal to
\begin{equation*}
    \tilde c_{k , i} := \frac{1}{\frac{1}{c_{k , i}} + \frac{|i|^4}{\beta}}.
\end{equation*}
Using~\eqref{eq:lowerboundconductancecijk}, these conductances can be lower bounded by
\begin{equation*}
    \tilde c_{k , i} \geq \frac{c \beta}{Nk^2 + |i|^4}.
\end{equation*}
Using that $k/2 \leq |i| \leq k$ and that $k \leq K \leq \lfloor \sqrt{N} \rfloor$, we obtain that $|i|^4 \leq k^2 K^2 \leq k^2 N$. Consequently
\begin{equation} \label{eq:lowerboundtildec}
    \tilde c_{k , i} \geq \frac{c \beta}{Nk^2}.
\end{equation}
In the last step of the proof, we merge all the edges connecting the left side of $-N$ to $0$ together (N.B. due to the Dirichlet boundary condition, all the vertices on the left on $-N$ can be identified together and this operation does not affect the distribution of the height function) and erase all the other edges. Note that this operation increases the variance of $\varphi(0)$. We refer to Figure~\ref{upperboundalpha=3lastpart.picture} for a visual description of this step. It generates a graph with one edge connecting the left side of $-N$ to $0$ whose conductance is equal to 
\begin{equation*}
    \sum_{\substack{k = 1 \\ k \, \mathrm{even}}}^{\sqrt{N}} \sum_{i \in \{ -k , \ldots, -k/2\}} \tilde c_{k , i}.
\end{equation*}
Using~\eqref{eq:lowerboundtildec}, we can lower bound this conductance by
\begin{equation*}
    \sum_{\substack{k = 1 \\ k \, \mathrm{even}}}^{\sqrt{N}} \sum_{i \in \{ -k , \ldots, -k/2\}}  \tilde c_{k , i} \geq \frac{c \beta}{N} \sum_{\substack{k = 1 \\ k \, \mathrm{even}}}^{\sqrt{N}} \sum_{i \in \{ -k , \ldots, -k/2\}} \frac{1}{k^2} \geq \frac{c \beta}{N} \sum_{\substack{k = 1 \\ k \, \mathrm{even}}}^{\sqrt{N}} \frac{1}{k} \geq \frac{c \beta \ln N}{N}.
\end{equation*}
The variance of $\varphi(0)$ is thus upper bounded by the one of an integer-valued Gaussian random variable with conductance equal to $c \beta \ln N / N$, which is itself bounded by $C N /  (\beta \ln N)$, see Lemma \ref{lemma:lemma2.2}.

\subsubsection{The case $\alpha \in (1,3)$ via B\"{a}umler's technique} \label{sec.Baumlertechnique}

This section is devoted to the proof of the upper bounds of Theorem~\ref{thm.gaussian} for $\alpha \in (1, 3)$. We note that the argument written below is a (minor) adaptation of the proof of B\"{a}umler~\cite[Proof of Theorem 1.1]{baumler2023recurrence} of the recurrence of the one-dimensional random walk with long-range jumps (i.e., the argument is the same but has been adapted to match the formalism used in this article). We refer to Figures~\ref{Splittingupperbound.picture} and~\ref{isolatingupperbound.picture} and~\ref{mergingupperbound.picture} for guidance.

    \begin{figure} 
\begin{center}
\includegraphics[scale=0.7]{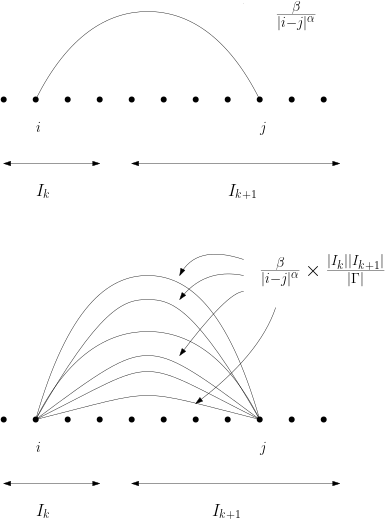}
\caption{First part of the proof: each edge connecting two vertices $i \in I_k$ and $j \in I_{k+1}$ is split into $\left| \Gamma \right| / (|I_k| |I_{k+1}|)$ parallel edges.} \label{Splittingupperbound.picture}
\end{center}
\end{figure}

    \begin{figure} 
\begin{center}
\includegraphics[scale=0.65]{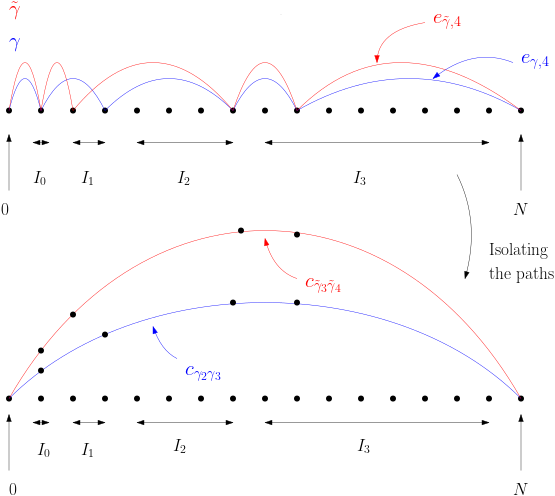}
\caption{On the top picture: Two paths, $\gamma$ and $\tilde \gamma$ of $\Gamma$ are drawn (in blue and red respectively). They shall only use distinct edges. On the bottom picture: Proposition~\ref{corollary.identification} is used to separate the path and generate direct lines connecting the vertex $0$ to the region $\left\{ n \in \N \, : \, n \geq N \right\}$ where the Dirichlet boundary condition is imposed} \label{isolatingupperbound.picture}
\end{center}
\end{figure}

    \begin{figure} 
\begin{center}
\includegraphics[scale=0.55]{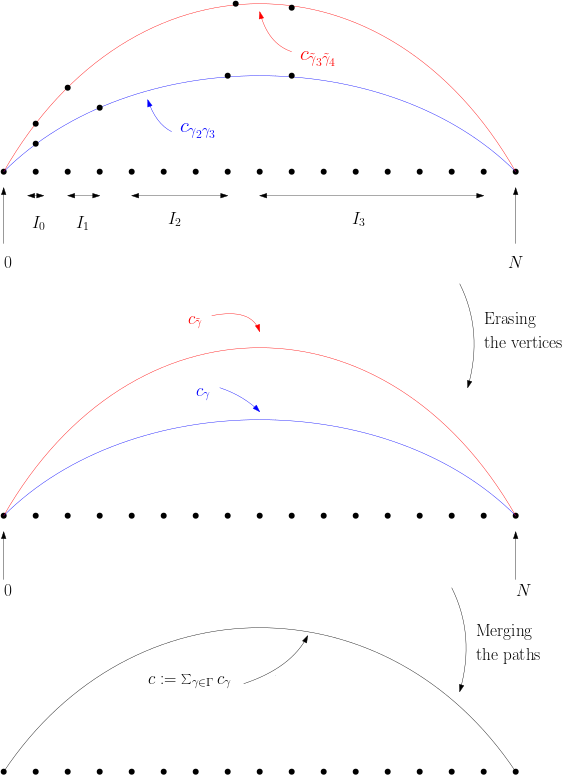}
\caption{Last step of the proof : all the edges connecting the vertex $0$ to the region $\left\{ n \in \N \, : \, n \geq N \right\}$ are merged together and the conductances of all the other edges are reduced to $0$ (i.e., the edges are deleted).} \label{mergingupperbound.picture}
\end{center}
\end{figure}

We first let $K$ be the largest integer such that $2^{K+1} \leq N$ (i.e., $K := \lfloor \ln_2 N \rfloor - 1$). We then consider the following collection of discrete intervals: for any $k \in \{ 0 , \ldots, K\}$,
\begin{equation*}
    I_k := \{ 2^k , \ldots, 2^{k+1} - 1\}.
\end{equation*}
We then denote by $\Gamma$ the collection of finite paths
\begin{equation*}
    \Gamma := \left\{ \gamma : \{0 , \ldots, K\} \to \N \, : \, \forall k \in \{ 1 , \ldots, K\},~ \gamma_k \in I_k \right\}.
\end{equation*}
Note that $\Gamma$ is a finite set and we denote by $\left| \Gamma \right|$ its cardinality (N.B. it is equal to $2^{K(K+1)/2}$).

We then perform the following operations which all increase the variance of the height~$\varphi(0)$ (see Figures~\ref{Splittingupperbound.picture},~\ref{isolatingupperbound.picture} and~\ref{mergingupperbound.picture} for guidance):

\begin{itemize}
    \item For each integer $k \in \{ 0 , \ldots, K -1 \}$ and each pair of vertices $i \in I_k, j \in I_{k+1}$, we split the edge connecting $i$ and $j$ into $\frac{|\Gamma|}{|I_k| |I_{k+1}|}$ edges (N.B. this is the number of paths of $\Gamma$ which visit both $i$ and $j$), and assign to each new edge the conductance 
    \begin{equation*}
        c_{i j} := \frac{\beta |I_k| |I_{k+1}|}{|\Gamma| |i-j|^{\alpha}} \geq c \beta \frac{2^{(2 - \alpha)k}}{|\Gamma|} .
    \end{equation*}
    This step is depicted on Figure~\ref{Splittingupperbound.picture}.
    \item For each path $\gamma \in \Gamma$ and each integer $k \in \{ 0 , \ldots, K \}$, we select an edge connecting the vertices $\gamma_k$ and $\gamma_{k+1}$ and denote this edge by $e_{\gamma, k}$. We choose these edges so that they are all distinct (N.B. this is possible thanks to the previous operation). 
    \item We isolate the collection of edges $\left\{ e_{\gamma, k} \, : \, k \in \{0 , \ldots, K\} \right\}$ from the rest of the chain (see Figure~\ref{isolatingupperbound.picture}).
    \item We then erase the vertices on the line which has been isolated at the previous step (see Figure~\ref{mergingupperbound.picture}). These operations generate an edge connecting $0$ to $\gamma_K$ (with $\gamma_K > N$ the endpoint of $\gamma$) whose conductance is equal to
    \begin{equation*}
        c_\gamma := \frac{1}{\sum_{k = 1}^K \frac{1}{c_{\gamma_k  \gamma_{k+1}}}} \geq  \frac{c \beta}{\left| \Gamma \right| \sum_{k = 1}^K 2^{(\alpha - 2)k}} \geq 
        \left\{ 
        \begin{aligned}
        c \left| \Gamma \right|^{-1} \beta & ~\mbox{if}~ \alpha \in (1,2), \\
        \frac{c  \left| \Gamma \right|^{-1} \beta}{\ln N} & ~\mbox{if}~ \alpha = 2, \\
        c  \left| \Gamma \right|^{-1} \beta N^{(2- \alpha)} & ~\mbox{if}~ \alpha \in (2 , 3).
        \end{aligned}
        \right.
    \end{equation*}
    \item We apply the previous operations to all the paths $\gamma \in \Gamma$, then merge together the $\left| \Gamma \right|$ edges obtained and add their conductance. This generates an edge connecting $0$ to the region $\left\{ x \in \Z \, : \, x \geq N \right\}$ (once again, due to the Dirichlet boundary condition, all these vertices can be identified together) whose conductance is given by
\begin{equation*}
    \sum_{\gamma \in \Gamma} c_\gamma \geq \left\{ 
        \begin{aligned}
        c \beta & ~\mbox{if}~ \alpha \in (1,2), \\
        \frac{c \beta}{\ln N} & ~\mbox{if}~ \alpha = 2, \\
        c \beta N^{(2- \alpha)} & ~\mbox{if}~ \alpha \in (2 , 3).
        \end{aligned}
        \right.
\end{equation*}
We finally reduce to $0$ the values of all the other conductances.
\end{itemize}

All the operations performed until now increase the variance of the height $\varphi(0)$, and they show that this variance is smaller than the one of an integer-valued Gaussian random variable with a conductance equal to the constant $c$ above. Applying Lemma~\ref{lemma:lemma2.2}, we obtain
\begin{equation*}
        \mathrm{Var}_{N ,\beta , \alpha} \left[ \varphi(0) \right] \leq 
        \left\{ \begin{aligned}
            C_\beta & ~\mbox{for} ~ \alpha \in (1,2), \\
            C_\beta \ln N & ~\mbox{for} ~ \alpha = 2, \\
            C_\beta N^{\alpha - 2} & ~\mbox{for} ~ \alpha \in (2,3).
        \end{aligned} \right.
    \end{equation*}

\section{Delocalisation for the integer-valued $q$-SOS long-range chain} \label{sec.section4qSOS}

This section is devoted to the proof of Theorem~\ref{thm.qSOS}. The strategy is similar to the one used in Section~\ref{sec:section3} with an additional step: following the recent article of Sellke~\cite{sellke2024localization}, we first use that the potential $x \mapsto e^{- |x|^q}$ can be decomposed into a mixture of centred Gaussian densities (see Proposition~\ref{prop:Gaussdecomp}) together with the FKG inequality to rephrase the question of the delocalisation of the $q$-SOS long-range chain into the question of the delocalisation of the Gaussian chain equipped with i.i.d. random conductances, which can then be analysed using similar techniques as in the proof of Theorem~\ref{thm.gaussian}. These random conductances follow a stable distribution of parameter $q/2 \in (0,1)$ and thus possess a heavy-tail which affects the behaviour of the chain (and modifies the growth exponent).

We note that, while an explicit dependence of the constants in the inverse temperature $\beta$ could be obtained from the argument, this dependence is more difficult to follow than in the previous section (where all the constants were linear in $\beta$) and we thus decided not to keep track of it. We will thus from now on allow the constants to depend on the parameter $\beta$ (in addition to the exponents $\alpha$ and $q$).

To simplify the notation, we set in all this section
$$
\beta_q := \beta^{2/q} ~~\mbox{and}~~ \alpha_q := \frac{2 \alpha}{q}.$$ 

\subsection{The lower bound} \label{sec:lowerbound-q-SOS}

\subsubsection{Domination by a Gaussian long-range chain}

In this section, we relate the variance of the height of a $q$-SOS long-range chain to the one of a Gaussian long-range chain with random conductances. We first fix an exponent $q \in (0 , 2]$ and recall the definition of the probability distribution $\mu_q$ introduced in Proposition~\ref{prop:Gaussdecomp}.

\begin{definition}[Random conductances] \label{def.randomconduntances}
We let $(\lambda_{ij})_{i , j \in \Z}$ be i.i.d. random variables distributed according to the measure $\mu_q$. We denote by $\mathbf{P}_q$ the law of the conductances and by $\mathbf{E}_{q}$ the expectation with respect to $\mathbf{P}_q$.
\end{definition}

Let us fix an integer $N \in \N$, an exponent $\alpha$, an inverse temperature $\beta$, and a collection of conductances $\lambda=(\lambda_{ij})_{i , j \in \Z} \in (0 , \infty)^{\Z \times \Z}$. We denote by $ \mu_{N , \beta , \alpha}^\lambda$ the finite-volume integer-valued Gaussian free field with long-range interaction and conductances $\lambda$. We denote by $\mathrm{Var}_{N , \beta , \alpha}^\lambda$ the variance with respect to $\mu_{N ,  \beta , \alpha}^\lambda$.

\begin{proposition}[Domination by a Gaussian chain with i.i.d. random conductances] \label{prop.dominationgaussianlower}
For any integer $N \in \N$, any inverse temperature $\beta\in (0, \infty)$ and any range exponent $\alpha \in (1 , \infty)$, one has the inequality
\begin{equation*}
    \mathrm{Var}_{N , \beta , \alpha}^{q-\mathrm{SOS}} \left[ \varphi(0) \right] \geq \mathbf{E}_{q} \left[ \mathrm{Var}_{N ,\beta_{q} , \alpha_{q}}^\lambda \left[ \varphi(0) \right] \right].
\end{equation*}
\end{proposition}

\begin{proof}
    We first write
    \begin{equation*}
        \mathrm{Var}_{N , \beta , \alpha}^{q-\mathrm{SOS}} \left[ \varphi(0) \right] = \frac{\sum_{\varphi \in \Omega_N} \left| \varphi(0) \right|^2 \exp \left( - \beta \sum_{i , j \in \Z} \frac{\left| \varphi(i) - \varphi(j) \right|^q}{|i - j|^\alpha} \right)}{Z_{N , \beta , \alpha}^{q-\mathrm{SOS}}}.
    \end{equation*}
    We next use the identity stated in Proposition~\ref{prop:Gaussdecomp} to rewrite the numerator of the previous display as follows
    \begin{align*}
        \lefteqn{\sum_{\varphi \in \Omega_N} \left| \varphi(0) \right|^2 \exp \left( - \beta \sum_{i , j \in \Z} \frac{\left| \varphi(i) - \varphi(j) \right|^q}{|i - j|^\alpha} \right) } \qquad & \\ 
        & = \sum_{\varphi \in \Omega_N} \left| \varphi(0) \right|^2 \prod_{i , j \in \Z}  \exp \left( - \beta \frac{\left| \varphi(i) - \varphi(j) \right|^q}{|i - j|^\alpha} \right) \\
        & = \sum_{\varphi \in \Omega_N} \left| \varphi(0) \right|^2 \prod_{i , j \in \Z} \int_{(0 , \infty)} \exp \left( - \lambda  \beta_{q} \frac{\left| \varphi(i) - \varphi(j) \right|^2}{|i - j|^{\alpha_{q}}} \right) d \mu_q(\lambda) \\
        & = \sum_{\varphi \in \Omega_N} \left| \varphi(0) \right|^2 \int_{(0 , \infty)^{\Z \times \Z}} \exp \left( - \sum_{i , j \in \Z}  \frac{\lambda_{ij}  \beta_{q} \left| \varphi(i) - \varphi(j) \right|^2}{|i - j|^{\alpha_{q}}} \right) \prod_{i,j \in \Z} d \mu_q(\lambda_{ij}).
    \end{align*}
    We then use of the identity
    \begin{equation*}
        \sum_{\varphi \in \Omega_N} \left| \varphi(0) \right|^2  \exp \left( - \sum_{i , j \in \Z}  \frac{\lambda_{ij}  \beta_{q} \left| \varphi(i) - \varphi(j) \right|^2}{|i - j|^{\alpha_{q}}} \right) =  \mathrm{Var}_{N ,\beta_{q} , \alpha_{q}}^\lambda \left[ \varphi(0) \right] \times  Z_{N ,\beta_{q} , \alpha_{q}}^\lambda,
    \end{equation*}
    together with Definition~\ref{def.randomconduntances} to obtain
    \begin{equation*}
        \sum_{\varphi \in \Omega_N} \left| \varphi(0) \right|^2 \exp \left( - \beta \sum_{i , j \in \Z} \frac{\left| \varphi(i) - \varphi(j) \right|^q}{|i - j|^\alpha} \right) = \mathbf{E}_q \left[  \mathrm{Var}_{N ,\beta_{q} , \alpha_{q}}^\lambda \left[ \varphi(0) \right] \times  Z_{N ,\beta_{q} , \alpha_{q}}^\lambda \right].
    \end{equation*}
    The same computation yields the identity for the partition function
    \begin{equation*}
        Z_{N , \beta , \alpha}^{q-\mathrm{SOS}} = \sum_{\varphi \in \Omega_N} \exp \left( - \beta \sum_{i , j \in \Z} \frac{\left| \varphi(i) - \varphi(j) \right|^q}{|i - j|^\alpha} \right) = \mathbf{E}_q \left[  Z_{N ,\beta_{q} , \alpha_{q}}^\lambda \right].
    \end{equation*}
    Combining the few previous results, we have obtained
    \begin{equation*}
        \mathrm{Var}_{N , \beta , \alpha}^{q-\mathrm{SOS}} \left[ \varphi(0) \right] =  \frac{\mathbf{E}_q \left[  \mathrm{Var}_{N ,\beta_{q} , \alpha_{q}}^\lambda \left[ \varphi(0) \right] \times  Z_{N ,\beta_{q} , \alpha_{q}}^\lambda \right]}{\mathbf{E}_q \left[  Z_{N ,\beta_{q} , \alpha_{q}}^\lambda \right]}.
    \end{equation*}
    We finally note that the two functions $\lambda \mapsto \mathrm{Var}_{N ,\beta_{q} , \alpha_{q}}^\lambda \left[ \varphi(0) \right]$ and $\lambda \mapsto Z_{N ,\beta_{q} , \alpha_{q}}^\lambda$ are both decreasing in the conductances $\lambda = (\lambda_{ij})_{i, j \in \Z}$. This result is a consequence of Proposition~\ref{prop:monotonicityconductances} and Remark~\ref{rem:partitionfunctionmonotone}. We may thus apply the FKG-inequality stated in Proposition~\ref{prop.FKGinequality} to obtain the inequality.
    \begin{equation*}
        \mathrm{Var}_{N , \beta , \alpha}^{q-\mathrm{SOS}} \left[ \varphi(0) \right] =  \frac{\mathbf{E}_q \left[  \mathrm{Var}_{N ,\beta_{q} , \alpha_{q}}^\lambda \left[ \varphi(0) \right] \times  Z_{N ,\beta_{q} , \alpha_{q}}^\lambda \right]}{\mathbf{E}_q \left[  Z_{N ,\beta_{q} , \alpha_{q}}^\lambda \right]} \geq \mathbf{E}_q \left[  \mathrm{Var}_{N ,\beta_{q} , \alpha_{q}}^\lambda \left[ \varphi(0) \right]  \right].
    \end{equation*}
    The proof of Proposition~\ref{prop.dominationgaussianlower} is complete.
\end{proof}

The following sections contain the proofs of the lower bounds of Theorem~\ref{thm.qSOS}. They are obtained by combining the technique developed in Section~\ref{sec:lowerboundGaussian} together with the results of the previous section. As in Section~\ref{sec:lowerboundGaussian}, we split the argument into different cases, depending on the value of the exponent $\alpha$.

\subsubsection{The cases $\alpha > 2+{q}/{2}$ } \label{sec:sec421}

Consider a random i.i.d. collection $\lambda=(\lambda_{ij})_{i,j \in \Z}$ distributed under $\mu_q$, and work with the quenched model with conductances $\lambda$.

After performing the same operations as in Section \ref{sec:lowerboundGaussianalpha>3}, we obtain a (nearest neighbour) line $\{ 0 , \ldots, N \}$, and, for any vertex $i \in \{ -N, \ldots, N - 1 \}$, the value of the conductance on the edge $i (i +1)$ is equal to 
\begin{equation*}
C_i := \beta_q \sum_{j \leq i} \sum_{j' \geq i + 1} \frac{\lambda_{jj'}}{\left| j - j'\right|^{\alpha_q - 1}} + \beta_q \sum_{j \leq - i-1} \sum_{j' \geq - i} \frac{\lambda_{jj'}}{\left| j - j'\right|^{\alpha_q - 1}}.
\end{equation*}
We note that the random variables $(C_i)_{i \in \{0 , \ldots, N-1\}}$ are identically distributed but not independent.

The variance of $\varphi(0)$ can then be estimated using Remark~\ref{rem:partitionfunctionmonotone} together with the observation that the variance of an integer-valued Gaussian random variable of conductance $C_i$ is larger than $c e^{-C_i}$. We thus obtain 
\begin{equation*}
\mathrm{Var}^\lambda_{N , \beta_q , \alpha_q} \left[ \varphi(0) \right] \geq  
\sum_{i=1}^N c e^{- C_i }.
\end{equation*}
Taking the expectation with respect to the conductances and applying Proposition~\ref{prop.dominationgaussianlower}, we deduce that
\begin{equation*}
    \mathrm{Var}_{N , \beta , \alpha}^{q-\mathrm{SOS}} \left[ \varphi(0) \right]  \geq \mathbf{E}_q \left[ \mathrm{Var}^\lambda_{N , \beta_q , \alpha_q} \left[ \varphi(0) \right] \right] \geq \sum_{i=1}^N c \mathbf{E}_q \left[ e^{- C_i } \right].
\end{equation*}
To prove that this sum is of order $N$, we will show that there exists a constant $K_{1/2}<\infty$ such that, for any $i\in\{0 , \ldots, N-1\}$,
\begin{equation} \label{eq:lowqSOScondit}
\mathbf{P}_q \left[ C_i<K_{1/2} \right]> \frac{1}{2}.
\end{equation}
(N.B. Since the random variables $(C_i)_{i \in \{1 , \ldots, N\}}$ are identically distributed, the left-hand side does not depend on the vertex $i$.)
Let $$S:= \sum_{j \geq 2, j' \leq 1 } \frac{\lambda_{jj'}}{\left|j - j'\right|^{\alpha_q -1}}. $$ Then its Laplace transform satisfies  
$$L_S(t):=\mathbf{E}_q [e^{-tS} ]=\prod_{j \geq 2, j' \leq 1 } L_\lambda(t/\left|j - j'\right|^{\alpha_q -1}),$$
where $L_\lambda$ is the Laplace transform of the law $\mu_q$, i.e., $L_\lambda(t)= e^{-{t}^{q/2}}$. We thus have
$$L_S(t)=\prod_{j \geq 2, j' \leq 1 } \exp \left( -{t}^{q/2}|j-j'|^{-(\alpha_q -1)q/2} \right) = \exp \left( -{t}^{q/2}\sum_{j \geq 2, j' \leq 1 } \left|j - j'\right|^{-(\alpha_q -1)q/2} \right).$$
Under the assumption $\alpha > 2+q/2$, we have $2 < (\alpha_q -1)q/2 $ which implies that the sum in the exponential on the right-hand side of the previous display is finite, i.e.,
\begin{equation*}
    C_{\alpha, q} := \sum_{j \geq 2, j' \leq 1 } \left|j - j'\right|^{-(\alpha_q -1)q/2}  < \infty,
\end{equation*}
and thus $L_S(t)=e^{-C_{\alpha, q} {t}^{q/2}}$.
We deduce that if $\alpha > 2+q/2$ then the law of the random variable $S / C_{\alpha, q}^{2/q}$ is the measure $\mu_q$ (as they have the same Laplace transform):
$$\frac{S}{C_{\alpha, q}^{2/q}} \sim \mu_q.$$
Thus, by Remark~\ref{remark2.13}, we have, for any $K\geq 1$,
$$\mathbf{P}_q \left[ S<K \right] \geq 1- C K^{-q/2}.$$
Since the random variable $C_i$ is the sum of two random variables whose law is the same as the one of $S$ (multiplied by the factor $\beta_q$), we may combine the previous inequality with a union bound to deduce that there exists a constant $K_{1/2} < \infty$ (which depends on the parameters $\alpha, \beta , q$) such that the inequality~\eqref{eq:lowqSOScondit} is satisfied.
We then deduce that
$$\mathrm{Var}_{N , \beta , \alpha}^{q-\mathrm{SOS}} \left[ \varphi(0) \right] \geq \sum_{i=1}^N c \mathbf{E}_q \left[ e^{-C_i} \right]
\geq \sum_{i=1}^N c \mu_q(C_i<K_{1/2}) e^{-K_{1/2}}
\geq c N.$$

\subsubsection{The cases $\alpha \in (2, 2+{q}/{2})$}
The argument is similar to the one presented in Section~\ref{sec:section312}, we will thus use the notation introduced there and only present the main differences.
We set 
\begin{equation} \label{eq:0939gamma}
    \gamma := \frac{2}{q} \left( 2 - \alpha \right)+ 1 \in (0,1)
\end{equation}
and proceed as in the Gaussian case of Section~\ref{sec:section312} by first constructing blocks of length $N^\gamma$ and estimating the values of the various conductances (as depicted in Figure~\ref{Delocalphasmallerthan3.picture}). As in the proof of Section~\ref{sec:section312}, we focus on the interval $\{ 0 , \ldots, \lceil N^\gamma \rceil \}$. We first consider the edge connecting the two extremities $0$ and $\lceil N^\gamma \rceil$ and note that the value of the conductance of this edge is given by the formula
\begin{equation*}
    \beta_q \sum_{j \leq 0 , j' \geq \lceil N^\gamma \rceil} \frac{(n_{jj'} +1) \lambda_{jj'}}{|j - j'|^{\alpha_q}}.
\end{equation*}
We then consider an integer $k \in \{ 1 , \ldots, \lceil N^\gamma \rceil  \}$ and note that the value of the conductance on the edge connecting $k$ to $\lceil N^\gamma \rceil$ is given by the identity
\begin{equation*}
    \beta_{q} \sum_{j \geq \lceil N^\gamma \rceil} \frac{(n_{kj} +1) \lambda_{kj}}{|k - j|^{\alpha_q}}.
\end{equation*}
We next use Proposition~\ref{corollary.identification} to identify all the vertices $\{ 0 , \ldots, \lceil N^\gamma \rceil - 1 \}$ together (as depicted in Figure~\ref{Delocalphasmallerthan31.picture}). This operation transforms the interval $\{ 0 , \ldots, \lceil N^\gamma \rceil \}$ into a graph containing two vertices connected by a conductance whose value is 
\begin{equation} \label{eq:14140711}
    C_0 := \beta_{q} \underset{\eqref{eq:14140711}-(i)}{\underbrace{\sum_{j \leq 0 , j' \geq \lceil N^\gamma \rceil} \frac{(n_{jj'} +1) \lambda_{ij}}{|j - j'|^{\alpha_q}}}} + \beta_{q} \underset{\eqref{eq:14140711}-(ii)}{\underbrace{\sum_{k = 1}^{\lceil N^\gamma \rceil - 1} \sum_{j \geq \lceil N^\gamma \rceil} \frac{(n_{kj} +1) \lambda_{kj}}{|j - k|^{\alpha_q}}}}.
\end{equation}
The next step of the argument is to show that the random conductance $C_0$ does not take too large values. Specifically, we will show that there exists a constant $C < \infty$ (depending on $\alpha, \beta$ and $q$) such that for any $K \geq 1$,
\begin{equation} \label{eq:1526}
    \mathbf{P}_q \left[ C_0 < K \right] \geq 1 - C K^{-q/2}.
\end{equation}
We first estimate the term~\eqref{eq:14140711}-(i) using the inequality~\eqref{eq:estnij} to bound the term $n_{ij}$ together with the trivial bound $n_{ij} \leq C N^{1-\gamma}$ (i.e., we bound $n_{ij}$ by the maximal number of blocks of size $\lceil  N^\gamma \rceil$ in the interval $\{-N , \ldots, N\}$). We obtain
\begin{equation} \label{eq:14270711}
    \eqref{eq:14140711}-(i) \leq \frac{C}{N^\gamma} \sum_{\substack{j \in \{ -N, \dots, 0\}, \\  j'\in \{ \lceil N^\gamma \rceil, \dots, N \}}} \frac{\lambda_{jj'}}{\left|j - j'\right|^{\alpha_q -1}} + C N^{1-\gamma}\sum_{\substack{ j \leq 0, j' \geq \lceil N^\gamma \rceil \\ j \leq -N \, \mathrm{or} \, j' \geq N}} \frac{\lambda_{jj'}}{\left|j - j'\right|^{\alpha_q }}
\end{equation}
Let  $$S_\gamma := \frac{1}{N^\gamma} \sum_{\substack{j \in \{ -N, \dots, 0\}, \\  j'\in \{\lceil N^\gamma \rceil, \dots, N \}}} \frac{\lambda_{jj'}}{\left|j - j'\right|^{\alpha_q -1}} + N^{1 - \gamma} \sum_{\substack{ j \leq 0, j' \geq \lceil N^\gamma \rceil \\ j \leq -N \, \mathrm{or} \, j' \geq N}} \frac{\lambda_{jj'}}{\left|j - j'\right|^{\alpha_q}}.$$
Then 
\begin{align} \label{eq:4.9}
L_{S_\gamma}(t)&
= \prod_{\substack{j \in \{ -N, \dots, 0\}, \\  j'\in \{\lceil N^\gamma \rceil, \dots, N \}}} L_\lambda \left( N^{-\gamma} t \left|j - j'\right|^{-(\alpha_q -1)} \right) \times  \prod_{\substack{ j \leq 0, j \geq \lceil N^\gamma \rceil \\ j \leq -N \, \mathrm{or} \, j \geq N}} L_\lambda \left( N^{1 - \gamma} t \left|j - j'\right|^{-\alpha_q} \right) \\
&= \exp \left({-{t}^{q/2} N^{-\gamma q/2} \sum_{\substack{j \in \{ -N, \dots, 0\}, \\  j'\in \{\lceil N^\gamma \rceil, \dots, N \}}} \left|j - j'\right|^{-(\alpha_q -1)q/2}}  - t^{q/2}  N^{(\alpha - 2)} \sum_{\substack{ j \leq 0, j' \geq \lceil N^\gamma \rceil \\ j \leq -N \, \mathrm{or} \, j' \geq N}} \left|j - j'\right|^{-\alpha} \right), \notag
\end{align}
where we used the identities $(1-\gamma)q/2 = \alpha - 2$ and $\alpha_q q/2 = \alpha$ in the last line.
The value of the exponent~$\gamma$ has been chosen in~\eqref{eq:0939gamma} so that there exists a constant $C_{\alpha, q} \in (1 , \infty)$ (independent of $N$) such that
\begin{equation*}
     N^{-\gamma q/2} \sum_{\substack{j \in \{ -N, \dots, 0\}, \\  j'\in \{\lceil N^\gamma \rceil, \dots, N \}}} \left|j - j'\right|^{-(\alpha_q -1)q/2}  +  N^{(\alpha - 2)} \sum_{\substack{ j \leq 0, j' \geq \lceil N^\gamma \rceil \\ j \leq -N \, \mathrm{or} \, j' \geq N}} \left|j - j'\right|^{-\alpha} \leq C_{\alpha, q}.
\end{equation*}
We thus deduce that the random variable $S_\gamma / C_{\alpha, q}^{2/q}$ (multiplied by a suitable non-negative number smaller than $1$) is distributed according to $\mu_q$. In particular, we have the tail estimates: for any $K\geq 1$,
\begin{equation*}
    \mathbf{P}_q\left[ S_\gamma <K \right] \geq 1- C K^{-q/2}.
\end{equation*}
Combining the previous inequality with~\eqref{eq:14270711}, we deduce that
\begin{equation} \label{eq:firstermqSOSlowerbound}
    \mathbf{P}_q \left[ \sum_{j \leq 0 , j' \geq \lceil N^\gamma \rceil} \frac{(n_{jj'} +1) \lambda_{jj'}}{|j - j'|^{\alpha_q}} < K \right] \geq 1 - C K^{-q/2}.
\end{equation}
The term~\eqref{eq:14140711}-(ii) can be estimated using a similar argument: we first use the inequality~\eqref{eq:estnij} and the bound $n_{ij} \leq C N^{1-\gamma}$ to write
\begin{equation*}
    \eqref{eq:14140711}-(ii) \leq C \sum_{k = 1}^{\lceil N^\gamma \rceil - 1}  \left( \sum_{j = \lceil N^\gamma \rceil}^{2\lceil N^\gamma \rceil-1} \frac{ \lambda_{kj}}{|j - k|^{\alpha_q}} + \frac{1}{N^\gamma}\sum_{j = 2\lceil N^\gamma \rceil}^{N-1} \frac{ \lambda_{kj}}{|j - k|^{\alpha_q-1}} + N^{1-\gamma}\sum_{j \geq N} \frac{ \lambda_{kj}}{|j - k|^{\alpha_q}} \right).
\end{equation*}
Denoting by $\tilde S_\gamma$ the random variable on the right-hand side, we may compute its Laplace transform and obtain
\begin{equation*}
    L_{\tilde S_\gamma}(t) = \exp \left(-t^{q/2} \sum_{k = 1}^{\lceil N^\gamma \rceil - 1} \left( \sum_{j = \lceil N^\gamma \rceil}^{2\lceil N^\gamma \rceil-1} |j - k|^{-\alpha}  +   N^{-\frac{\gamma q}{2}} \sum_{j = 2\lceil N^\gamma \rceil}^{N-1} \left|j - k\right|^{- \frac{(\alpha_q -1)q}{2}}  -  N^{(\alpha - 2)} \sum_{j \geq N} \left|j - k\right|^{-\alpha}  \right) \right).
\end{equation*}
The value of the exponent $\gamma$ has been selected so that the term inside the exponential is bounded uniformly in $N$. This implies the following tail estimate on the random variable $\tilde S_\gamma$, for any $K\geq 1$,
\begin{equation*}
    \mathbf{P}_q [\tilde S_\gamma <K] \geq 1- C K^{-q/2},
\end{equation*}
and consequently
\begin{equation*}
\mathbf{P}_q \left[ \sum_{k = 1}^{\lceil N^\gamma \rceil - 1} \sum_{j \geq \lceil N^\gamma \rceil} \frac{(n_{kj} +1) \lambda_{kj}}{|j - k|^{\alpha_q}} < K \right] \geq 1 - C K^{-q/2}.
\end{equation*}
Combining the previous inequality with~\eqref{eq:firstermqSOSlowerbound} completes the proof of~\eqref{eq:1526}.

The end of the proof is essentially identical to the one written in the Gaussian case in Section~\ref{sec:section312} and in the case of the exponent $\alpha > 2 + q/2 $ written in Section~\ref{sec:sec421}: the problem has been reduced to the question of the delocalisation of a nearest neighbour Gaussian chain of length $N^{1-\gamma} = N^{\frac{2}{q}(\alpha - 2)}$ with random conductance satisfying the tail estimate~\eqref{eq:1526} which can be handled by first folding the chain as in Figure~\ref{Delocalphalarger33.picture} and then applying Remark~\ref{rem:partitionfunctionmonotone} as in Section~\ref{sec:sec421}.

\subsubsection{The case $\alpha = 2+{q}/{2}$}

This case is essentially identical to the previous one: the only difference is that we replace the term $N^\gamma$ by the term $(\ln N)^{2/q}$ (N.B. the reason for the exponent $2/q$ is that, when taking the Laplace transform as in~\eqref{eq:4.9}, an exponent $q/2$ will be added to the term $(\ln N)^{-2/q}$ so as to obtain a factor $(\ln N)^{-1}$ which is the correct renormalisation for the sum). We thus omit the technical details.

\subsubsection{The cases $\alpha \in (1,2]$}

This case is the simplest: using the monotonicity of the variance $\mathrm{Var}_{N , \alpha, \beta}^\lambda \left[ \varphi(0) \right]$ in the parameter $N$ (which is a consequence of Proposition~\ref{corollary.identification} even in the presence of arbitrary conductances), we may write
\begin{equation*}
    \mathrm{Var}_{N , \beta , \alpha}^{q-\mathrm{SOS}} \left[ \varphi(0) \right]  \geq \mathbf{E}_q \left[ \mathrm{Var}^\lambda_{N , \beta_q , \alpha_q} \left[ \varphi(0) \right] \right] \geq \mathbf{E}_q \left[ \mathrm{Var}^\lambda_{1 , \beta_q , \alpha_q} \left[ \varphi(0) \right] \right].
\end{equation*}
For the chain of length $1$, the distribution of the height $\varphi(0)$ is an integer-valued Gaussian random variable with conductance equal to $\beta_q \sum_{i \in \Z \setminus \{0\}} \lambda_{0i}/ |i|^{\alpha_q}$ (which is almost surely finite). We thus obtain
\begin{equation*}
 \mathrm{Var}_{N , \beta , \alpha}^{q-\mathrm{SOS}} \left[ \varphi(0) \right]  \geq \mathbf{E}_q \left[ \mathrm{Var}^\lambda_{1 , \beta_q , \alpha_q} \left[ \varphi(0) \right] \right] \geq c \mathbf{E}_q \left[ e^{\beta_q \sum_{i \in \Z \setminus \{0\}} \lambda_{0i}/ |i|^{\alpha_q}} \right] \geq c_\beta.
\end{equation*}

\subsection{The upper bound} \label{sec:upperbound-q-SOS}

This section is devoted to the upper bound of Theorem~\ref{thm.qSOS} and is split into two subsections. We first show that the variance of the height of the $q$-SOS long-range chain can be dominated from above by the one of a Gaussian chain with i.i.d. random conductances (distributed according to the probability distribution $\tilde \mu_q$ defined below) which can then be estimated using the same technique as the one used for the Gaussian chain with range exponent $\alpha \in (1,3)$ presented in Section~\ref{sec.Baumlertechnique}.

\subsubsection{Domination by a Gaussian long-range chain}

\begin{definition}[Random conductances and the measure $\tilde \mu_q$] \label{def.tildmuq}
For $q \in (0,2)$, we denote by $\tilde \mu_q$ the Borel probability measure on $( 0, \infty)$ given by the identity 
\begin{equation*}
    \tilde \mu_q(d\lambda) := \frac{\lambda^{-1/2}}{\int_0^\infty  \lambda^{-1/2} \mu_q(d\lambda)}{\mu_q(d\lambda)}.
\end{equation*}
We let $(\lambda_{ij})_{i , j \in \Z}$ be i.i.d. random variables distributed according to the measure $\tilde \mu_q$. We denote by $ \mathbf{\tilde E}_{q}$ the expectation with respect to the conductances $(\lambda_{ij})_{i , j \in \Z}$.
\end{definition}

We then prove the following inequality showing that the $q$-SOS long-range chain is more localised than a Gaussian random chain with independent conductances distributed according to the measure $\tilde \mu_q$. 

\begin{proposition}[Domination by a Gaussian chain with i.i.d. random conductances] \label{prop.dominationgaussian}
For any integer $N \in \N$, one has the inequality
\begin{equation*} 
    \mathrm{Var}_{N , \beta , \alpha}^{q-\mathrm{SOS}} \left[ \varphi(0) \right] \leq \mathbf{\tilde E}_{q}  \left[ \mathrm{Var}_{N ,\beta_q , \alpha_q}^\lambda \left[ \varphi(0) \right] \right].
\end{equation*}
\end{proposition}

\begin{remark} \label{rem:rem4.6}
    We note that the tail of the law $\tilde \mu_q$ decays faster than the one of the measure $\mu_q$ (both tails decay as a power law, but not with the same exponent). This is the reason why the exponents in the upper and lower bounds of Theorem~\ref{thm.qSOS} are not equal. We further believe that, if the stochastic dominations stated in Propositions~\ref{prop.dominationgaussianlower} and~\ref{prop.dominationgaussian} could be improved so as to obtain dominations from above and below by Gaussian chains with i.i.d. random conductances such that the tails of the laws of the conductances decay at the same rate, then it would be possible to obtain matching exponents in the upper and lower bounds of Theorem~\ref{thm.qSOS}.
\end{remark}

\begin{proof}
We start from the identity
    \begin{align*}
    \mathrm{Var}_{N , \beta , \alpha}^{q-\mathrm{SOS}} \left[ \varphi(0) \right] & =  \frac{\int_{(0 , \infty)^{\Z \times \Z}}  \mathrm{Var}_{N ,\beta_q , \alpha_q}^\lambda \left[ \varphi(0) \right] \times  Z_{N ,\beta_q , \alpha_q}^\lambda \prod_{i,j \in \Z} d \mu_q(\lambda_{ij}) }{\int_{(0 , \infty)^{\Z \times \Z}} Z_{N ,\beta_q , \alpha_q}^\lambda \prod_{i,j \in \Z} d \mu_q(\lambda_{ij})} \\
    & = \frac{\mathbf{E}_q \left[  \mathrm{Var}_{N ,\beta_q , \alpha_q}^\lambda \left[ \varphi(0) \right] \times  Z_{N ,\beta_q , \alpha_q}^\lambda \right]}{\mathbf{E}_q \left[  Z_{N ,\beta_q , \alpha_q}^\lambda \right]}.
    \end{align*}
On a formal level, the strategy is to use the definition of the measure $\tilde \mu_q$ to rewrite the previous identity as follows
\begin{equation*}
    \mathrm{Var}_{N , \beta , \alpha}^{q-\mathrm{SOS}} \left[ \varphi(0) \right] =  \frac{\mathbf{ \tilde E}_q \left[  \mathrm{Var}_{N ,\beta_q , \alpha_q}^\lambda \left[ \varphi(0) \right] \times  (Z_{N ,\beta_q , \alpha_q}^\lambda \times \prod_{i,j\in \Z} \sqrt{\lambda_{ij}}) \right]}{\mathbf{\tilde E}_q \left[  Z_{N ,\beta_q , \alpha_q}^\lambda \times \prod_{i,j \in \Z} \sqrt{\lambda_{ij}} \right]},
\end{equation*}
to show that the function $\lambda \mapsto  Z_{N ,\beta_q , \alpha_q}^\lambda \times \prod_{i,j \in \Z} \sqrt{\lambda_{ij}}$ is increasing in $\lambda$ and to apply the FKG inequality (with one function decreasing and one increasing) to conclude. 

In order to make the previous argument rigorous, one needs to make sense of the infinite product $\prod_{i,j \in \Z} \sqrt{\lambda_{ij}}$. This is achieved below by using an approximation argument.

Let us fix a large integer $M \in \N$ (whose value will be sent to infinity later). Given a collection of conductances $\lambda := (\lambda_{ij})_{i , j \in \Z \times \Z}$, we denote by $\lambda^M := (\lambda^M_{ij})_{i,j \in \Z}$ the collection of conductances defined by the identities $\lambda^M_{ij} = \lambda_{ij}$ if $i , j \in \{- M , \ldots, M\}$ and $\lambda^M_{ij} = 0$ otherwise. We then consider the ratio of expectations
\begin{equation*}
    \frac{\mathbf{E}_q \left[  \mathrm{Var}_{N ,\beta_q , \alpha_q}^{\lambda^M} \left[ \varphi(0) \right] \times  Z_{N ,\beta_q , \alpha_q}^{\lambda^M} \right]}{\mathbf{E}_q \left[  Z_{N ,\beta_q , \alpha_q}^{\lambda^M} \right]}.
\end{equation*}
It is easy to show that this ratio converges to the variance $\mathrm{Var}_{N , \beta , \alpha}^{q-\mathrm{SOS}} \left[ \varphi(0) \right]$ (as all the quantities involved are monotone in $M$). Additionally, since there are now only finitely many random variables involved, we may use the definition of the measure $\tilde \mu_q$ to write
\begin{equation} \label{eq:identityforFKG}
     \frac{\mathbf{E}_q \left[  \mathrm{Var}_{N ,\beta_q , \alpha_q}^{\lambda^M} \left[ \varphi(0) \right] \times  Z_{N ,\beta_q , \alpha_q}^{\lambda^M} \right]}{\mathbf{E}_q \left[  Z_{N ,\beta_q , \alpha_q}^{\lambda^M} \right]} = \frac{\mathbf{\tilde E}_q \left[  \mathrm{Var}_{N ,\beta_q , \alpha_q}^{\lambda^M} \left[ \varphi(0) \right] \times  (Z_{N ,\beta_q , \alpha_q}^{\lambda^M} \times \prod_{i,j\in \{ - M, \ldots, M \}} \sqrt{\lambda_{ij}}) \right]}{\mathbf{\tilde E}_q \left[  Z_{N ,\beta_q , \alpha_q}^{\lambda^M} \times \prod_{i,j\in \{ - M, \ldots, M \}} \sqrt{\lambda_{ij}} \right]}.
\end{equation}
We then show that the function $\lambda := (\lambda_{ij})_{i , j \in \Z} \mapsto Z_{N ,\beta_q , \alpha_q}^{\lambda^M} \times  \prod_{i,j \in \{-M, \ldots, M\}} \sqrt{\lambda_{ij}}$ is increasing and first we note that this property is equivalent to the following statement
\begin{equation} \label{eq:16161809}
    \forall \lambda  \in (0,\infty)^{\Z \times \Z}, ~ \forall k , l \in \{-M , \ldots , M \}, ~~ \frac{\partial }{\partial \lambda_{kl}}  \ln \left( Z_{N ,\beta_q , \alpha_q}^{\lambda^M} \times  \prod_{i,j \in \{-M , \ldots, M\}} \sqrt{\lambda_{ij}} \right) \geq 0.
\end{equation}
To prove~\eqref{eq:16161809}, we fix $\lambda \in (0,\infty)^{\Z \times \Z}$, fix $k , l \in \{- M , \ldots, M\}$ and write
\begin{align} \label{eq:17451809}
    \frac{\partial }{\partial \lambda_{kl}}  \ln \left( Z_{N ,\beta_q , \alpha_q}^{\lambda^M} \times  \prod_{i,j \in \{-M , \ldots, M\}} \sqrt{\lambda_{ij}} \right) & = \frac{\partial }{\partial \lambda_{kl}}  \ln \left( Z_{N ,\beta_q , \alpha_q}^{\lambda^M} \right) +  \frac{\partial }{\partial \lambda_{kl}} \ln \sqrt{\lambda_{kl}} \\
    & = \frac{1}{Z_{N ,\beta_q , \alpha_q}^{\lambda^M}} \times \left( \frac{\partial}{\partial \lambda_{kl}} Z_{N ,\beta_q, \alpha_q}^{\lambda^M} \right) + \frac{1}{2 \lambda_{kl}}. \notag
\end{align}
For the first term on the right-hand side, an explicit computation gives the bound
\begin{equation} \label{eq:17461809}
    \frac{1}{Z_{N ,\beta_q , \alpha_q}^{\lambda^M}} \times \left( \frac{\partial}{\partial \lambda_{kl}} Z_{N ,\beta_q , \alpha_q}^{\lambda^M} \right) = - \frac{\beta_q}{|k - l|^{\alpha_q}} \mathrm{Var}_{N ,\beta_q , \alpha_q}^{\lambda^M} \left[ \varphi(k) - \varphi(l) \right].
\end{equation}
We then reduce the value of all the conductances $\lambda_{ij}$ with $(i , j) \neq (k , l)$ to $0$ and make two observations. First, by Proposition~\ref{prop:monotonicityconductances}, this operation increases that value of the variance of $\varphi(k) - \varphi(l)$. After performing this operation, the law of $\varphi(k) - \varphi(l)$ is an integer-valued Gaussian random variable of conductance $\beta_q \lambda_{kl}/|k - l|^{\alpha_q}$. Applying Lemma~\ref{lemma:lemma2.2}, we have thus obtained
\begin{equation*}
        \mathrm{Var}_{N ,\beta_q , \alpha_q}^{\lambda^M} \left[ \varphi(k) - \varphi(l) \right] \leq \frac{1}{2} \frac{|k - l|^{\alpha_q}}{\beta_q \lambda_{kl}}.
\end{equation*}
Combining the previous inequality with~\eqref{eq:17451809}, we deduce that
\begin{equation*}
    \frac{1}{Z_{N ,\beta_q , \alpha_q}^{\lambda^M}} \times \left( \frac{\partial}{\partial \lambda_{kl}} Z_{N ,\beta_q , \alpha_q}^{\lambda^M} \right) \geq  - \frac{1}{2 \lambda_{kl}},
\end{equation*}
which, by~\eqref{eq:17451809}, implies~\eqref{eq:16161809}. Now that~\eqref{eq:16161809} has been established, we may apply the FKG-inequality to the right-hand side of~\eqref{eq:identityforFKG} and obtain that
\begin{align*}
    \mathrm{Var}_{N , \beta , \alpha}^{q-\mathrm{SOS}} \left[ \varphi(0) \right] & = \lim_{M \to \infty} \frac{\mathbf{E}_q \left[  \mathrm{Var}_{N ,\beta_q , \alpha_q}^{\lambda^M} \left[ \varphi(0) \right] \times  Z_{N ,\beta_q , \alpha_q}^{\lambda^M} \right]}{\mathbf{E}_q \left[  Z_{N ,\beta_q , \alpha_q}^{\lambda^M} \right]} \\
    & \leq \lim_{M \to \infty} \mathbf{\tilde E}_{q}  \left[ \mathrm{Var}_{N ,\beta_q , \alpha_q}^{\lambda^M} \left[ \varphi(0) \right] \right] \\
    & = \mathbf{\tilde E}_{q}  \left[ \mathrm{Var}_{N ,\beta_q , \alpha_q}^{\lambda} \left[ \varphi(0) \right] \right] 
\end{align*}
The proof of Proposition~\ref{prop.dominationgaussian} is complete.
\end{proof}

\subsubsection{Bäumler's technique again}

In the section, we readapt the argument of Section~\ref{sec.Baumlertechnique} to the case of the Gaussian chain with i.i.d. random conductances. The argument below is a minor adaptation of the proof presented there (which is itself a minor adaptation of the proof of~\cite[Proof of Theorem 1.1]{baumler2023recurrence}).

Let us fix a (large) integer $N \in \N$ and let $K$ be the largest integer such that $2^{K+1} \leq N$. As in the proof of Section~\ref{sec.Baumlertechnique}, we introduce the intervals
\begin{equation*}
    I_k := \{ 2^k , \ldots, 2^{k+1} - 1\} ~~\mbox{for}~~ k \in \{ 0 , \ldots, K\}.
\end{equation*}
as well as the collections of paths
\begin{equation*}
    \Gamma := \left\{ \gamma : \{0 , \ldots, K\} \to \N \, : \, \forall k \in \{ 1 , \ldots, K\},~ \gamma_k \in I_k \right\}.
\end{equation*}

We perform the following operations (referring to Figures~\ref{Splittingupperbound.picture} and~\ref{isolatingupperbound.picture} and~\ref{mergingupperbound.picture} for guidance):

\begin{itemize}
    \item For each integer $k \in \{ 0 , \ldots, K -1 \}$ and each pair of vertices $i \in I_k, j \in I_{k+1}$, we split the edge connecting $i$ and $j$ into $\frac{|\Gamma|}{|I_k| |I_{k+1}|}$ edges to which we assign the conductance 
    \begin{equation*}
        c_{i j} := \frac{\beta_q \lambda_{ij} |I_k| |I_{k+1}|}{|\Gamma| |i-j|^{ \alpha_q}} \leq C \lambda_{ij} \frac{2^{(2 - \alpha_q) k}}{|\Gamma|}.
    \end{equation*}
    \item For each path $\gamma \in \Gamma$ and each integer $k \in \{ 0 , \ldots, K \}$, we select an edge connecting the vertices $\gamma_k$ and $\gamma_{k+1}$ and denote this edge by $e_{\gamma, k}$. We choose these edges so that they are all distinct. 
    \item We then isolate the collection of edges $\left\{ e_{\gamma, k} \, : \, k \in \{0 , \ldots, K\} \right\}$ from the rest of the chain and erase the vertices on the line which has been isolated. These operations generate an edge connecting $0$ to $\gamma_K$ (with $\gamma_K > N$) whose conductance is equal to
    \begin{equation*}
        c_\gamma := \frac{1}{\sum_{k = 1}^K \frac{1}{c_{\gamma_k  \gamma_{k+1}}}} \geq  \frac{c }{\left| \Gamma \right| \sum_{k = 1}^K \frac{2^{(\alpha_q-2) k}}{\lambda_{\gamma_k \gamma_{k+1}}}}.
    \end{equation*}
    \item We apply the previous operations to all the paths $\gamma \in \Gamma$, then merge the $\left| \Gamma \right|$ edges together and add their conductance. This generates an edge connecting $0$ to the region $\left\{ n \in \N \, : \, n \geq N \right\}$ whose conductance is given by
\begin{equation*}
    c := \sum_{\gamma \in \Gamma} c_\gamma \geq \frac{c}{\left| \Gamma \right|}  \sum_{\gamma \in \Gamma} \frac{1}{\sum_{k = 1}^K \frac{2^{(\alpha_q-2) k}}{\lambda_{\gamma_k \gamma_{k+1}}}}.
\end{equation*}
We finally reduce to $0$ the values of all the other conductances.
\end{itemize}

All the operations performed until now increase the variance of the height $\varphi(0)$, and they show that this variance is smaller than the one of an integer-valued Gaussian random variable with a conductance equal to the constant $c$ above (whose variance is thus smaller than $1/c$). Specifically, we have the upper bound
\begin{equation*}
    \mathrm{Var}_{N ,\beta_{q} , \alpha_qq}^\lambda \left[ \varphi(0) \right] \leq C \frac{\left| \Gamma \right|}{\sum_{\gamma \in \Gamma} \frac{1}{\sum_{k = 1}^K \frac{2^{(\alpha_q-2) k}}{\lambda_{\gamma_k \gamma_{k+1}}}}}.
\end{equation*}
Using that  the harmonic mean is smaller than the arithmetic mean, we may simplify the previous inequality by writing
\begin{equation*}
    \mathrm{Var}_{N ,\beta_{q} , \alpha_q}^\lambda \left[ \varphi(0) \right] \leq \frac{C}{\left| \Gamma \right|} \sum_{\gamma \in \Gamma} \sum_{k = 1}^K \frac{2^{(\alpha_q-2) k}}{\lambda_{\gamma_k \gamma_{k+1}}}.
\end{equation*}
Taking the expectation, we further deduce that
\begin{align*}
    \mathbf{\tilde E}_{q}  \left[ \mathrm{Var}_{N ,\beta_q , \alpha_q}^\lambda \left[ \varphi(0) \right] \right] & \leq \frac{C}{ \left| \Gamma \right|} \sum_{\gamma \in \Gamma} \sum_{k = 1}^K 2^{(\alpha_q-2) k}  \mathbf{\tilde E}_q \left[ \frac{1}{\lambda_{\gamma_k \gamma_{k+1}}} \right] \\
    & \leq \frac{C}{\left| \Gamma \right|} \sum_{\gamma \in \Gamma} \sum_{k = 1}^K 2^{(\alpha_q-2) k} \\
    & \leq  C \sum_{k = 1}^K 2^{(\alpha_q-2) k} \\
    & \leq \left\{ \begin{aligned}
            C_\beta & ~\mbox{for} ~ \alpha < q, \\
            C_\beta \ln N & ~\mbox{for} ~ \alpha = q, \\
            C_\beta N^{\alpha_q - 2} & ~\mbox{for} ~ \alpha > q,
        \end{aligned} \right.
\end{align*}
where in the first inequality, we used the tail estimate (close to $0$) of the measure $\mu_q$ stated in Remark~\ref{remark2.13} (which implies that a similar tail estimate holds for the measure $\tilde \mu_q$), and where we used the definition of the integer $K$ in the last inequality.

We finally note that we always have (N.B. this is obtained by erasing all the long-range edges as in Section~\ref{sec:casealpha>3})
\begin{equation*}
    \mathbf{\tilde E}_{q}  \left[ \mathrm{Var}_{N ,\beta_q , \alpha_q}^\lambda \left[ \varphi(0) \right] \right] \leq C N.
\end{equation*}
Combining the two previous inequalities completes the proof of Theorem~\ref{thm.qSOS}.\\

\textbf{Acknowledgments.} We thank Christophe Garban for enlightening discussions, constant availability and appealing conjectures on this localisation/delocalisation problems. We thank Aernout van Enter for his encouragements and his feedback on a previous version of the paper.

{\small
\bibliographystyle{abbrv}
\bibliography{quantitative.bib}
}

\end{document}